\documentclass[11pt]{amsart}
\usepackage{mathrsfs}
\usepackage{amsfonts}
\usepackage{latexsym,amsmath,amssymb}

 \renewcommand{\epsilon}{\varepsilon}

\newtheorem{theorem}{Theorem}[section]

 \newtheorem{lemma}[theorem]{Lemma}
 
 \newtheorem{Corollary}[theorem]{Corollary}
 \newtheorem{proposition}[theorem]{proposition}
 \newtheorem{Proposition}[theorem]{Proposition}
\newtheorem{deff}[theorem]{Definition}
 \newtheorem{rem}[theorem]{Remark}
 \newcommand{\bth}{\begin{theorem}}
 \newcommand{\ble}{\begin{lemma}}
 \newcommand{\bcor}{\begin{corr}}
 \newcommand{\bdeff}{\begin{deff}}
 \newcommand{\bprop}{\begin{proposition}}
 \newcommand{\ele}{\end{lemma}}
 \newcommand{\ecor}{\end{corr}}
 \newcommand{\edeff}{\end{deff}}
 
 \newcommand{\eprop}{\end{proposition}}

 \renewcommand{\Pi}{\varPi}

 \renewcommand{\epsilon}{\varepsilon}

\numberwithin{equation}{section}

\thanks{The first author was supported  by National Science Foundation of China(No.12101145) and  Guangxi Science, Technology Project (Grant No. GuikeAD22035202).}

\title
[Nonlinear second boundary conditions ]{On the second boundary value problem for mean curvature flow
in Minkowski  space}
\author{Rongli Huang}
\address{School of Mathematics and Statistics, Guangxi Normal University,
Guilin, Guangxi 541004, People's Republic of China,
 E-mail: ronglihuangmath@gxnu.edu.cn}

\begin{document}
\maketitle
\begin{abstract}
This is a sequel to \cite{HRY} and \cite{CHY}, which study the second boundary value problems for  mean curvature flow.  Consequently, we construct the translating solitons with prescribed Gauss image in Minkowski space.

\end{abstract}

\let\thefootnote\relax

\section{Introduction}

In a series of papers \cite{HR}-\cite{CHY}, the first author, Y.H.Ye and J.J.Chen  were concerned with
 special Lagrangian evolution equations
\begin{equation}\label{e1.1}
\frac{\partial u}{\partial t}=F_\tau(\lambda(D^2 u)),\quad
 \mathrm{in}\quad \Omega_{T}=\Omega\times(0,T),
\end{equation}
associated with the second boundary value problem
\begin{equation}\label{e1.2}
Du(\Omega)=\tilde{\Omega},  \quad t>0,
\end{equation}
and the initial condition
\begin{equation}\label{e1.3}
u=u_{0},  \quad t=0,
\end{equation}
for given $F_{\tau}$, $u_{0}$, $\Omega$  and  $\tilde{\Omega}$. Specifically, $\Omega$, $\tilde{\Omega}$  are uniformly convex bounded
domains with smooth boundary in $\mathbb{R}^{n}$ and
\begin{equation*}
F_{\tau}(\lambda):=\left\{ \begin{aligned}
&\frac{1}{2}\sum_{i=1}^n\ln\lambda_{i}, &&\tau=0, \\
& \frac{\sqrt{a^2+1}}{2b}\sum_{i=1}^n\ln\frac{\lambda_{i}+a-b}{\lambda_{i}+a+b},  &&0<\tau<\frac{\pi}{4},\\
& -\sqrt{2}\sum_{i=1}^n\frac{1}{1+\lambda_{i}}, &&\tau=\frac{\pi}{4},\\
& \frac{\sqrt{a^2+1}}{b}\sum_{i=1}^n\arctan\frac{\lambda_{i}+a-b}{\lambda_{i}+a+b},  \ \ &&\frac{\pi}{4}<\tau<\frac{\pi}{2},\\
& \sum_{i=1}^n\arctan\lambda_{i}, &&\tau=\frac{\pi}{2}.
\end{aligned} \right.
\end{equation*}
Here $a=\cot \tau$, $b=\sqrt{|\cot^2\tau-1|}$,  $x=(x_{1},x_{2},\cdots,x_{n})$, $u=u(x,t)$ and $\lambda(D^2 u)=(\lambda_1,\cdots, \lambda_n)$ are the eigenvalues of Hessian matrix $D^2 u$ according to $x$. The long time asymptotic behavior of solutions were obtained.
\begin{theorem}\label{prop1.1}$($\cite{HR}-\cite{CHY}$)$.
Assume that $\Omega$, $\tilde{\Omega}$ are bounded, uniformly convex domains with smooth boundary in $\mathbb{R}^{n}$, $0<\alpha_{0}<1$ and
$0<\tau\leq\frac{\pi}{2}.$
  Then for any given initial function $u_{0}\in C^{2+\alpha_{0}}(\bar{\Omega})$
  which is   uniformly convex and satisfies $Du_{0}(\Omega)=\tilde{\Omega}$,  the  strictly convex solution of (\ref{e1.1})-(\ref{e1.3}) exists
  for all $t\geq 0$ and $u(\cdot,t)$ converges to a function $u^{\infty}(x,t)=u^\infty(x)+C_{\infty}\cdot t$ in $C^{1+\zeta}(\bar{\Omega})\cap C^{4+\alpha}(\bar{D})$ as $t\rightarrow\infty$
  for any $D\subset\subset\Omega$, $\zeta<1$,$0<\alpha<\alpha_{0}$, i.e.,

  $$\lim_{t\rightarrow+\infty}\|u(\cdot,t)-u^{\infty}(\cdot,t)\|_{C^{1+\zeta}(\bar{\Omega})}=0,\qquad
  \lim_{t\rightarrow+\infty}\|u(\cdot,t)-u^{\infty}(\cdot,t)\|_{C^{4+\alpha}(\bar{D})}=0.$$
And $u^{\infty}(x)\in C^{1+1}(\bar{\Omega})\cap C^{4+\alpha}(\Omega)$ is a solution of
\begin{equation}\label{e1.18}
\left\{ \begin{aligned}F_\tau(\lambda(D^2 u))&=C_{\infty},
&  x\in \Omega, \\
Du(\Omega)&=\tilde{\Omega}.
\end{aligned} \right.
\end{equation}
The constant $C_{\infty}$ depends only on $\Omega$, $\tilde{\Omega}$ and $F$. The solution to (\ref{e1.18}) is unique up to additions of constants.
\end{theorem}
\begin{rem}
By the regularity theory for parabolic partial differential equations, $\Omega$, $\tilde{\Omega}$  are uniformly convex bounded
domains with $C^{2+\alpha}$ boundary in $\mathbb{R}^{n}$, then the above theorem is also ture. For the convenience of readers, we always assume that the domains $\Omega$ and  $\tilde{\Omega}$  are uniformly convex bounded
domains with smooth boundary in the following setting.
\end{rem}
In particular, as far as $\tau=\frac{\pi}{2}$, Brendle-Warren's theorem \cite{SM} was proved by parabolic methods.

Inspired from Theorem 1.1, we consider the second boundary value problem for the mean curvature flow in  Minkowski space.
We start with the differential geometric, or "classical",
definition of the mean curvature flow for graphic hypersurfaces in  $\mathbb{R}^{n,1}$; for a nice introduction, see \cite{KG} and \cite{KE}.
Minkowski space $\mathbb{R}^{n,1}$ is the space $\mathbb{R}^{n}\times \mathbb{R}$ equipped with the metric
\begin{equation}\label{e1.5}
ds^{2}=dx^{2}_{1}+\cdots+dx^{2}_{n}-dx^{2}_{n+1}.
\end{equation}
We assume that $u\in C^{2}(\Omega)$ for some domain in $\mathbb{R}^{n}$  and $M$=graph $u$ over $\Omega$.
Denote $$D_{i}u=\dfrac{\partial u}{\partial x_{i}},
D_{ij}u=\dfrac{\partial^{2}u}{\partial x_{i}\partial x_{j}},
D_{ijk}u=\dfrac{\partial^{3}u}{\partial x_{i}\partial x_{j}\partial
x_{k}}, \cdots $$
and $|Du|=\sqrt{\sum_{i=1}^{n}|D_{i}u|^{2}}.$
It follows from \cite{LLJ}   that  we can state various geometric quantities associated with the graph of  $u\in C^{2}(\Omega)$.
In the coordinate system, Latin indices range from 1 to $n$ and indicate quantities in the graph. We use the Einstein summation convention,
if the indices are different from 1 to n.

We say that $M\subset \mathbb{R}^{n,1}$  is strictly spacelike, if $\sup_{\Omega}|Du|<1.$
The Minkowski metric (\ref{e1.5}) restricted to the strictly spacelike graph $M$ defines a
Riemannian metric on $M$, which in the standard coordinates on $\mathbb{R}^{n,1}$ is expressed by
\begin{equation}\label{e1.6}
g_{ij}=\delta_{ij}-D_{i}uD_{j}u,\quad 1\leq i,j\leq n.
\end{equation}
While the inverse of the induced metric and second fundamental form of $M$ are given by
\begin{equation}\label{e1.7}
g^{ij}=\delta_{ij}+\frac{D_{i}uD_{j}u}{1-|Du|^{2}},\quad 1\leq i,j\leq n,
\end{equation}
and, respectively,
\begin{equation}\label{e1.8}
h_{ij}=\frac{D_{ij}u}{\sqrt{1-|Du|^{2}}},\quad 1\leq i,j\leq n.
\end{equation}
The timelike unit normal vector field to $M$ is expressed by
\begin{equation}\label{e1.9}
\nu=\frac{(Du,1)}{\sqrt{1-|Du|^{2}}},\quad 1\leq i,j\leq n.
\end{equation}
The principal curvatures $\kappa_{1},\cdots,\kappa_{n}$ of $M$ are the eigenvalues of $[h_{ij}]$
relative to $[g_{ij}]$ and  the k-th mean curvature $H_{k}$ of $M$ is defined to be the k-th elementary
symmetric function of the principal curvatures,
\begin{equation}\label{e1.10}
H_{k}=S_{k}(\kappa_{1},\cdots,\kappa_{n})=\sum_{1\leq i_{1}<\cdots <i_{k}\leq n}\kappa_{i_{1}}\cdots\kappa_{i_{k}}.
\end{equation}
Specially, the mean curvature of $M$ is written as
\begin{equation}\label{e1.11}
H\triangleq H_{1}=\sum_{1\leq i\leq n}\kappa_{i}.
\end{equation}
Let $B_{1}(0)$ be the unit ball in $\mathbb{R}^{n}$  with the Klein model of the
hyperbolic geometry $\{(x,1)\in \mathbb{R}^{n,1}, |x|<1\}$.  Following Lemma 4.5 in \cite{CT}, the Gauss map of the graph
 $(x,u(x))$
is described from $\mathbb{R}^{n}$ to $B_{1}(0)$ as
$$\mathfrak{G}: \,\,x\mapsto Du(x).$$
We aim to construct convex spacelike  translating solitons with prescribed Gauss  image over any strictly convex domains by making use of the mean curvature flow.
Let $\Omega$ and $\tilde{\Omega}$ are two uniformly convex bounded domains with smooth boundary in  $\mathbb{R}^{n}$ and $\tilde{\Omega}\subset\subset B_{1}(0)$
respectively.
Along the lines of approach in previous work, including O.C. Schn$\ddot{\text{u}}$rer\cite{OC},
O.C. Schn$\ddot{\text{u}}$rer-K. Smoczyk \cite{OK}, J. Kitagawa \cite{JK}, Ma-Wang-Wei \cite{MWW}.
We consider a family of spacelike hypersurfaces
\begin{equation*}
X_{t}=X(\cdot,t):\,\,\Omega\mapsto \mathbb{R}^{n,1}
\end{equation*}
satisfying mean curvature flow in Minkowski  space
\begin{equation}\label{e1.11a}
\frac{\partial X}{\partial t}=H\nu,\quad
 \mathrm{in}\quad \Omega_{T}=\Omega\times(0,T)
\end{equation}
with the initial hypersurface
\begin{equation}\label{e1.11b}
X(\cdot,0)=X_{0}.
\end{equation}
For each $t>0$, we need the image of the Gauss map of $X(\cdot,t)$ over $\Omega$  to be $\tilde{\Omega}$.
That is
\begin{equation}\label{e1.11c}
\mathfrak{G}(\Omega)=\tilde{\Omega}.
\end{equation}
According to the property  of the mean curvature flow with prescribed Gauss image, one can see the papers \cite{Xin},\cite{MT}.

Suppose that each $X(\cdot,t)$ is the graph of a function $u(\cdot,t)$ and $X_{0}=(x,u_{0}(x))$. Then it follows from (\ref{e1.6})-(\ref{e1.11}) that we deduce that the equations
(\ref{e1.11a})-(\ref{e1.11c}) are equivalent up to diffeomorphisms in  $\Omega$ to
the following evolution equation
\begin{equation}\label{e1.12}
\frac{\partial u}{\partial t}=\sqrt{1-|Du|^{2}} \mathrm{div}(\frac{Du}{\sqrt{1-|Du|^{2}}}),\quad
 \mathrm{in}\quad \Omega_{T}=\Omega\times(0,T),
\end{equation}
associated with the second boundary value problem
\begin{equation}\label{e1.13}
Du(\Omega)=\tilde{\Omega},  \quad t>0,
\end{equation}
and the initial condition
\begin{equation}\label{e1.14}
u=u_{0},  \quad t=0.
\end{equation}
We are now in condition to state the theorems proved in this paper.
\begin{theorem}\label{t1.2}
Assume that $\Omega$, $\tilde{\Omega}$ are bounded, uniformly convex domains with smooth boundary in $\mathbb{R}^{n}$
and $\tilde{\Omega}\subset\subset B_{1}(0)$, $0<\alpha_{0}<1$.
  Then for any given initial function $u_{0}\in C^{2+\alpha_{0}}(\bar{\Omega})$
  which is   uniformly convex and satisfies $Du_{0}(\Omega)=\tilde{\Omega}$,  the  strictly convex solution of (\ref{e1.12})-(\ref{e1.14}) exists
  for all $t\geq 0$ and $u(\cdot,t)$ converges to a function $u^{\infty}(x,t)=u^\infty(x)+C_{\infty}\cdot t$ in $C^{1+\zeta}(\bar{\Omega})\cap C^{4+\alpha}(\bar{D})$ as $t\rightarrow\infty$
  for any $D\subset\subset\Omega$, $\zeta<1$,$0<\alpha<\alpha_{0}$, i.e.,

  $$\lim_{t\rightarrow+\infty}\|u(\cdot,t)-u^{\infty}(\cdot,t)\|_{C^{1+\zeta}(\bar{\Omega})}=0,\qquad
  \lim_{t\rightarrow+\infty}\|u(\cdot,t)-u^{\infty}(\cdot,t)\|_{C^{4+\alpha}(\bar{D})}=0.$$
And $u^{\infty}(x)\in C^{1+1}(\bar{\Omega})\cap C^{4+\alpha}(\Omega)$ is a solution of
\begin{equation}\label{e1.14aa}
\left\{ \begin{aligned}\mathrm{div}(\frac{Du}{\sqrt{1-|Du|^{2}}})&=\frac{C_{\infty}}{\sqrt{1-|Du|^{2}}},
&  x\in \Omega, \\
Du(\Omega)&=\tilde{\Omega}.
\end{aligned} \right.
\end{equation}
The constant $C_{\infty}$ depends only on $\Omega$, $\tilde{\Omega}$. The solution to (\ref{e1.14aa}) is unique up to additions of constants.
\end{theorem}
Higher regularity of the solutions obtained in Theorems 1.3  follows
from parabolic and elliptic regularity theory if the data are sufficiently smooth.

The rest of this article is organized as follows. The next section is to present the structure condition for the operator $F$
and  verify the short time existence of the mean curvature flow with the second boundary condition. Thus section 3 is devoted to carry out the strictly oblique estimate and the $C^2$ estimate according to the structure properties of the operators $G$ and $\tilde{G}$. Eventually, we give the long time existence and convergence of the parabolic flow in section 4.

\section{Preliminary}
There are ways to work around the short time existence of classical solutions of (\ref{e1.12})-(\ref{e1.14})(as discussed in \cite{HR}) so that
we have short term existence for the second boundary value problem. One can use the inverse function theorem in Fr$\acute{e}$chet spaces and the theory of linear parabolic equations for oblique boundary condition. The method is along the idea of proving  the short time existence of convex solutions on the second boundary value problem for Lagrangian mean curvature flow. The details can be seen
in the proof of Proposition 3.6 in \cite{WHB}.
\begin{Proposition}\label{p2.1}
According to the conditions in Theorem \ref{t1.2}, there exist some $T>0$ and $u\in C^{2+\alpha,\frac{2+\alpha}{2}}(\bar{\Omega}_{T})$ which depend only on $\Omega$, $\tilde{\Omega}$ and $u_0$, such that $u$ is a unique solution of (\ref{e1.12})-(\ref{e1.14}) and is uniformly convex in $x$ variable.
\end{Proposition}
  Based on the regularity theory for parabolic equations \cite{GM}, we assume that $u$ is strictly convex  solution to (\ref{e1.12})-(\ref{e1.14}) in the class $$C^{2+\alpha,1+\frac{\alpha}{2}}(\bar{\Omega}_{T})\cap C^{4+\alpha,2+\frac{\alpha}{2}}(\Omega_{T})$$ for some $T>0$.

Define $$F(\kappa_{1},\cdots, \kappa_{n})=\sum_{i=1}^n\kappa_{i},$$
and $$\Gamma^+_n:=\left\{(\kappa_1,\cdots,\kappa_n)\in \mathbb{R}^n:\kappa_i>0,\ i=1,\cdots,n\right\}.$$
Then $F$ is  a smooth symmetric function defined on $\bar{{\Gamma}^+_n}$ and satisfying

\begin{equation}\label{e1.2.0}
\sum_{i=1}^n\frac{\partial F}{\partial \kappa_i}\kappa_i\leq F,
\end{equation}
\begin{equation}\label{e1.2.a}
\frac{\partial F}{\partial \kappa_i}>0,\ \ 1\leq i\leq n\ \  \text{on}\ \  \Gamma^+_n,
\end{equation}
\begin{equation}\label{e1.2.ab}
\sum_{i=1}^n\frac{\partial F}{\partial \kappa_i}=n,\ \ 1\leq i\leq n\ \  \text{on}\ \  \Gamma^+_n,
\end{equation}
and
\begin{equation}\label{e1.2.1}
\left(\frac{\partial^2 F}{\partial \kappa_i\partial \kappa_j}\right)\leq 0\ \  \text{on}\ \  \bar{\Gamma^+_n}.
\end{equation}

Throughout the following, Einstein's convention of summation over repeated indices will be adopted.

The principal curvatures of $M\subset \mathbb{R}^{n,1}$ are the eigenvalues of the second fundamental form
$h_{ij}$ relative to $g_{ij}$, i.e, the eigenvalues of the mixed tensor $h^{j}_{i}\equiv h_{ik}g^{kj}$.
By \cite{LLJ} we remark that they are the eigenvalues of the symmetric matrix
\begin{equation}\label{e2.71}
a_{ij}=\frac{1}{v}b^{ik}D_{kl}ub^{lj},
\end{equation}
where $v=\sqrt{1-|Du|^{2}}$ and $b^{ij}$ is the positive square root of $g^{ij}$ taking the form
\begin{equation*}
b^{ij}=\delta_{ij}-\frac{D_{i}uD_{j}u}{v(1+v)}.
\end{equation*}
Explicitly we have shown
\begin{equation*}\label{e2.8}
\begin{aligned}
a_{ij}=&\frac{1}{v}\{D_{ij}u-\frac{D_{i}uD_{l}uD_{jl}u}{v(1+v)}-\frac{D_{j}uD_{l}uD_{il}u}{v(1+v)}\\
          &+\frac{D_{i}uD_{j}uD_{k}uD_{l}uD_{kl}u}{v^{2}(1+v)^{2}}\}.
\end{aligned}
\end{equation*}
Then by (\ref{e2.71}) one can deduce that
\begin{equation*}\label{e2.81}
D_{ij}u=vb_{ik}a_{kl}b_{lj}
\end{equation*}
where $b_{ij}$  is the inverse of $b^{ij}$ expressed as
\begin{equation*}
b_{ij}=\delta_{ij}+\frac{D_{i}uD_{j}u}{1+v}.
\end{equation*}

Denote $\mathcal{A}=[a_{ij}]$ and $F[\mathcal{A}]=\sum_{i=1}^n\kappa_{i}$, where $(\kappa_1,\cdots, \kappa_n)$ are the  eigenvalues of
the symmetric matrix $[a_{ij}]$. Then the properties of the operator $F$ are reflected in (\ref{e1.2.0})-(\ref{e1.2.1}).
It follows from (\ref{e1.2.a}) that we can show that
\begin{equation*}
F_{ij}[\mathcal{A}]\xi_{i}\xi_{j}>0 \quad \mathrm{for}\quad \mathrm{all}\quad \xi\in \mathbb{R}^{n}-\{0\}
\end{equation*}
where
\begin{equation*}
F_{ij}[\mathcal{A}]=\frac{\partial F[\mathcal{A}]}{\partial a_{ij}}.
\end{equation*}
From \cite{JS} we see that $[F_{ij}]$ diagonal if $\mathcal{\mathcal{A}}$ is diagonal, and in this case
\begin{equation*}
[F_{ij}]=\mathrm{diag}(\frac{\partial F}{\partial \kappa_{1}},\cdots, \frac{\partial F}{\partial \kappa_{n}}).
\end{equation*}
If $u$ is convex, by (\ref{e2.71}) we deduce that the  eigenvalues of the  matrix $[a_{ij}]$ must be in $\bar{\Gamma^+_n}$.
Then (\ref{e1.2.1}) implies  that
\begin{equation*}
F_{ij,kl}[\mathcal{A}]\eta_{ij}\eta_{kl}\leq 0
\end{equation*}
for any real symmetric matrix $[\eta_{ij}]$, where
\begin{equation*}
F_{ij,kl}[\mathcal{A}]=\frac{\partial^{2}F[\mathcal{A}]}{\partial a_{ij}\partial a_{kl}}.
\end{equation*}
According to the equation (\ref{e1.12}), we consider the nonlinear differential operators of the type
\begin{equation*}\label{e2.10}
u_{t}=G(Du,D^{2}u)
\end{equation*}
where $G(Du,D^{2}u)=\sqrt{1-|Du|^{2}}F[\mathcal{A}]$.
As in \cite{J}, differentiating this once we have
\begin{equation*}\label{e2.10}
D_{tk}u=G_{ij}D_{ijk}u+G_{i}D_{ik}u
\end{equation*}
where we use the notation
\begin{equation*}
G_{ij}=\frac{\partial G}{\partial r_{ij}},\quad G_{i}=\frac{\partial G}{\partial p_{i}}
\end{equation*}
with $r$ and $p$  representing for the second derivative and gradient variables respectively.
So as to prove the strict obliqueness estimate for the problem (\ref{e1.12})-(\ref{e1.14}), we need to recall some expressions from \cite{J} for the derivatives of G.  We have
\begin{equation}\label{e2.11}
G_{ij}=\sqrt{1-|Du|^{2}}F_{kl}\frac{\partial a_{kl}}{\partial r_{ij}}=b^{ik}F_{kl}b^{lj}
\end{equation}
and
\begin{equation*}\label{e2.12}
G_{i}=vF_{kl}\frac{\partial a_{kl}}{\partial p_{i}}-\frac{p_{i}}{v}F=vF_{kl}\frac{\partial}{\partial p_{i}}(\frac{1}{v}b^{kp}b^{ql})D_{pq}u-\frac{p_{i}}{v}F.
\end{equation*}
A simple calculation yields
\begin{equation*}\label{e2.13}
G_{i}=-\frac{D_{i}u}{v}F_{kl}a_{kl}-2F_{kl}a_{lm}b^{ik}D_{m}u-\frac{D_{i}u}{v}F.
\end{equation*}

We observe that $\mathcal{T}_{G}=\sum^{n}_{i=1}G_{ii}$ is the trace of a product of three matrices by (\ref{e2.11}), so it is invariant under orthogonal transformations. Hence, to compute $\mathcal{T}_{G}$, we may assume for now that $[a_{ij}]$ is diagonal. By virtue of (\ref{e1.13}) and $\tilde{\Omega}\subset\subset B_{1}(0)$, we obtain that $Du$ and $\frac{1}{v}$  are bounded. Then the eigenvalues of $[b^{ij}]$ are bounded between two controlled positive constants. Since (\ref{e2.11}), it follows that there exist positive constants $\sigma_{1}$,$\sigma_{2}$ depending only on the least upper bound of $|Du|$ in the set $\Omega$, such that
\begin{equation}\label{e2.14}
\sigma_{1}\mathcal{T}\leq\mathcal{T}_{G}\leq\sigma_{2}\mathcal{T}
\end{equation}
where $\mathcal{T}=\sum^{n}_{i=1}F_{ii}$.
By the concavity of $F$ and the positive definiteness of $[F_{ij}a_{ij}]$ imply that$[F_{ij}a_{ij}]$ is controlled by $F$, i.e.
\begin{equation*}\label{e2.15}
F_{ij}a_{ij}=\sum^{n}_{i=1}F_{i}\kappa_{i}\leq F(\kappa_{1},\ldots,\kappa_{n}).
\end{equation*}
Thus
\begin{equation}\label{e2.16}
|G_{i}|\leq \sigma_{3}F(\kappa_{1},\ldots,\kappa_{n})
\end{equation}
where $\sigma_{3}$ depends only on the known data.

By a direct computation and $\tilde{\Omega}\subset\subset B_{1}(0)$,  there exist two positive constants
$\sigma_{4}$,$\sigma_{5}$ depending only on $\Omega$ and $\tilde{\Omega}$, such that
\begin{equation}\label{e2.8}
\sigma_{4}\sum^{n}_{i=1}\frac{\partial F}{\partial \kappa_{i}}\kappa^{2}_{i}\leq \sum^{n}_{i=1}\frac{\partial G}{\partial \lambda_{i}}\lambda^{2}_{i}\leq \sigma_{5}\sum^{n}_{i=1}\frac{\partial F}{\partial \kappa_{i}}\kappa^{2}_{i}
\end{equation}
where $\lambda_1,\cdots, \lambda_n$ are the eigenvalues of Hessian matrix $D^2 u$ at $x \in \Omega$.

By making use of the methods on the second boundary value problems for equations of Monge-Amp\`{e}re type \cite{JU}, the parabolic boundary condition in (\ref{e1.13}) can be reformulated as
$$h(Du)=0,\qquad x\in \partial\Omega,\quad t>0,$$
where we need
\begin{deff}
A smooth function $h:\mathbb{R}^n\rightarrow\mathbb{R}$ is called the defining function of $\tilde{\Omega}$ if
$$\tilde{\Omega}=\{p\in\mathbb{R}^{n} : h(p)>0\},\quad |Dh|_{{\partial\tilde{\Omega}}}=1,$$
and there exists $\theta>0$ such that for any $p=(p_{1},\cdots, p_{n})\in \tilde{\Omega}$ and $\xi=(\xi_{1}, \cdots, \xi_{n})\in \mathbb{R}^{n}$,
$$\frac{\partial^{2}h}{\partial p_{i}\partial p_{j}}\xi_{i}\xi_{j}\leq -\theta|\xi|^{2}.$$
\end{deff}
Thus the mean curvature flow (\ref{e1.12})-(\ref{e1.14}) is equivalent to the evolution problem
\begin{equation}\label{e2.16}
\left\{ \begin{aligned}\frac{\partial u}{\partial t}&=G(Du,D^{2}u),\ \ && t>0,\  x\in \Omega, \\
h(Du)&=0,&& t>0,\  x\in\partial\Omega,\\
 u&=u_{0}, &&  t=0,\  x\in \Omega.
\end{aligned} \right.
\end{equation}
We can also define $\tilde{h}$ as the defining function of $\Omega$. That is,
$$\Omega=\{\tilde{p}\in\mathbb{R}^{n} : \tilde{h}(\tilde{p})>0\},\ \ \ |D\tilde{h}|_{\partial\Omega}=1, \ \ \ D^2\tilde{h}\leq -\tilde{\theta}I,$$
where $\tilde{\theta}$ is some positive constant. For each $t$, we will use the Legendre transform of $u$ which is the convex function $\tilde{u}$ on $\tilde{\Omega}=Du(\Omega)$ defined by
\begin{equation*}
\tilde{u}(y,t)=x\cdot Du(x,t)-u(x,t)
\end{equation*}
and
\begin{equation*}
y=Du(x,t).
\end{equation*}
It follows that
\begin{equation*}
\frac{\partial \tilde{u}}{\partial y_{i}}=x_{i},\frac{\partial ^{2}\tilde{u}}{\partial y_{i}\partial y_{j}}=u^{ij}(x)
\end{equation*}
where $[u^{ij}]=[D^{2}u]^{-1}$. Then $\tilde{u}$ satisfies
\begin{equation}\label{e2.17}
\left\{ \begin{aligned}
\frac{\partial \tilde{u}}{\partial t}&=\tilde{G}(y,D^{2}\tilde{u}),\ \ && T>t>0,\ \tilde{x}\in \tilde{\Omega}, \\
\tilde{h}(D\tilde{u})&=0,&& T>t>0,\ \tilde{x}\in\partial\tilde{\Omega},\\
 \tilde{u}&=\tilde{u}_{0}, &&  t=0,\ \tilde{x}\in \tilde{\Omega},
\end{aligned} \right.
\end{equation}
where $\tilde{G}(y,D^{2}\tilde{u})=-G(y,D^{2}\tilde{u}^{-1})$, and $\tilde{u}_{0}$ is the Legendre transformation of $u_{0}$. The unit inward normal vector of $\partial\Omega$ can be expressed by $\nu=D\tilde{h}$. For the same reason, $\tilde{\nu}=Dh$, where $\tilde{\nu}=(\tilde{\nu}_{1}, \tilde{\nu}_{2},\cdots,\tilde{\nu}_{n})$ is the unit inward normal vector of $\partial\tilde{\Omega}$.

\section{The strict obliqueness estimate and the $C^2$ estimate}
For the convenience, we denote $\beta=(\beta^{1}, \cdots, \beta^{n})$ with $\beta^{i}:=h_{p_{i}}(Du)$, and $\nu=(\nu_{1},\cdots,\nu_{n})$ as the unit inward normal vector at $x\in\partial\Omega$. The expression of the inner product is
\begin{equation*}
\langle\beta, \nu\rangle=\beta^{i}\nu_{i}.
\end{equation*}
\begin{lemma}\label{l3.1}($\dot{u}$-estimates) Assume that $\Omega$, $\tilde{\Omega}$ are bounded, uniformly convex domains with smooth boundary in $\mathbb{R}^{n}$ and  $\tilde{\Omega}\subset\subset B_{1}(0)$, $0<\alpha_{0}<1$, $u_{0}\in C^{2+\alpha_{0}}(\bar{\Omega})$
  which is   uniformly convex and satisfies $Du_{0}(\Omega)=\tilde{\Omega}$. If the strictly convex solution to (\ref{e2.16}) exists, then
\begin{equation*}
\min_{\bar{\Omega}}G(Du_{0}, D^{2}u_{0})\leq\dot{u}\leq\max_{\bar{\Omega}}G(Du_{0}, D^{2}u_{0}),
\end{equation*}
where $\dot{u}:=\frac{\partial u}{\partial t}$.
\end{lemma}
\begin{proof} From (\ref{e2.16}),  differentiating the equation yields
$$\frac{\partial(\dot{u}) }{\partial t}-G_{ij}\partial_{ij}(\dot{u})-G_{i}\partial_{i}(\dot{u})=0.$$
Using the maximum principle, we deduce that
$$\min_{\bar{\Omega}_{T}}(\dot{u}) =\min_{\partial\bar{\Omega}_{T}}(\dot{u}).$$
Without loss of generality, we assume that $\dot{u}\neq constant$. If there exists $x_{0}\in \partial\Omega$, $t_{0}>0$, such that $\dot{u}(x_{0},t_{0})=\min_{\bar{\Omega}_{T}}(\dot{u})$. On the one hand, since $\langle\beta, \nu\rangle>0$, by the Hopf Lemma (cf.\cite{LL1},\cite{LL2}) for parabolic equations, there must hold in the following
$$\dot{u}_{\beta}(x_{0},t_{0})\neq 0.$$
On the other hand,  we differentiate the function on the boundary condition and then obtain
$$\dot{u}_{\beta}=h_{p_{k}}(Du)\frac{\partial \dot{u}}{\partial x_{k}}=\frac{\partial h(Du)}{\partial t}=0.$$
It is a contradiction. So we get that
$$\dot{u}\geq \min_{\bar{\Omega}_{T}}(\dot{u})
=\min_{\partial\bar{\Omega}_{T}|_{t=0}}(\dot{u})=\min_{\bar{\Omega}}\left(G(Du_{0}, D^{2}u_{0})\right)$$
For the same reason, we have
$$\dot{u}\leq \max_{\bar{\Omega}_{T}}(\dot{u})
=\max_{\partial\bar{\Omega}_{T}|_{t=0}}(\dot{u})=\max_{\bar{\Omega}}\left(G(Du_{0}, D^{2}u_{0})\right).$$
Putting these facts together, the assertion follows.
\end{proof}
For the same reason, we also obtain the esttimates of dual type.
\begin{lemma}\label{l3.2}($\dot{\tilde{u}}$-estimates) Assume that $\Omega$, $\tilde{\Omega}$ are bounded, uniformly convex domains with smooth boundary in $\mathbb{R}^{n}$ and  $\tilde{\Omega}\subset\subset B_{1}(0)$, $0<\alpha_{0}<1$, $u_{0}\in C^{2+\alpha_{0}}(\bar{\Omega})$
  which is   uniformly convex and satisfies $Du_{0}(\Omega)=\tilde{\Omega}$. If the strictly convex solution to (\ref{e2.17}) exists, then
\begin{equation*}
\min_{\bar{\Omega}}\tilde{G}(y, D^{2}\tilde{u}_{0})\leq\dot{\tilde{u}}\leq\max_{\bar{\Omega}}\tilde{G}(Du_{0}, D^{2}\tilde{u}_{0}),
\end{equation*}
where $\dot{\tilde{u}}:=\frac{\partial \tilde{u}}{\partial t}$.
\end{lemma}
With this in hand, we can obtain the following structure conditions of the operator $F$ as $u$ is the strictly convex solution of (\ref{e2.16}).
\begin{lemma}\label{l3.3} Assume that $\Omega$, $\tilde{\Omega}$ are bounded, uniformly convex domains with smooth boundary in $\mathbb{R}^{n}$ and  $\tilde{\Omega}\subset\subset B_{1}(0)$, $0<\alpha_{0}<1$, $u_{0}\in C^{2+\alpha_{0}}(\bar{\Omega})$
  which is   uniformly convex and satisfies $Du_{0}(\Omega)=\tilde{\Omega}$. If the strictly convex solution to (\ref{e2.16}) exists. Then there exist positive constants $\Lambda_1$ and $\Lambda_2$, depending  on $u_{0}$,  such that there holds
  \begin{equation}\label{e3.1}
 \Lambda_1\leq\sum^{n}_{i=1}\frac{\partial F}{\partial \kappa_{i}}\leq \Lambda_2,
 \end{equation}
and
\begin{equation}\label{e3.2}
\Lambda_1\leq\sum^{n}_{i=1}\frac{\partial F}{\partial \kappa_{i}}\kappa^{2}_{i}\leq \Lambda_2.
\end{equation}
\end{lemma}
\begin{proof}
 The operator is of the form  $G(Du,D^{2}u)=\sqrt{1-|Du|^{2}}F[\mathcal{A}]$. We observe that
 $$F=\sum_{i=1}^n\kappa_{i}$$ and $$\frac{\partial F}{\partial \kappa_{i}}=1.$$
Then $$\sum^{n}_{i=1}\frac{\partial F}{\partial \kappa_{i}}=n$$ and $$\min_{p\in \tilde{\Omega}}\frac{1}{\sqrt{1-|p|^{2}}}\min_{\bar{\Omega},t\geq 0}G(Du,D^{2}u)\leq  F\leq \max_{p\in \tilde{\Omega}}\frac{1}{\sqrt{1-|p|^{2}}}\max_{\bar{\Omega},t\geq 0}G(Du,D^{2}u).$$
Using Lemma \ref{l3.1} and the equation (\ref{e2.16}) we deduce that  $$\min_{p\in \tilde{\Omega}}\frac{1}{\sqrt{1-|p|^{2}}}\min_{\bar{\Omega}}G(Du_{0},D^{2}u_{0})\leq  \sum_{i=1}^n\kappa_{i}\leq \max_{p\in \tilde{\Omega}}\frac{1}{\sqrt{1-|p|^{2}}}\max_{\bar{\Omega}}G(Du_{0},D^{2}u_{0}).$$
Combining $$\sum_{i=1}^n\kappa^{2}_{i}\leq(\sum_{i=1}^n\kappa_{i})^{2}\leq n\sum_{i=1}^n\kappa^{2}_{i}$$
with $$\sum^{n}_{i=1}\frac{\partial F}{\partial \kappa_{i}}\kappa^{2}_{i}=\sum^{n}_{i=1}\kappa^{2}_{i},$$
we obtain the desired results.
\end{proof}
By the above arguments,  we readily deduce  the strucure conditions from (\ref{e2.14})-(\ref{e2.8}) for the operator $G$.
\begin{Corollary}\label{c3.3}
 Assume that $\Omega$, $\tilde{\Omega}$ are bounded, uniformly convex domains with smooth boundary in $\mathbb{R}^{n}$ and  $\tilde{\Omega}\subset\subset B_{1}(0)$, $0<\alpha_{0}<1$, $u_{0}\in C^{2+\alpha_{0}}(\bar{\Omega})$
  which is   uniformly convex and satisfies $Du_{0}(\Omega)=\tilde{\Omega}$. If the strictly convex solution to (\ref{e2.16}) exists.
There   exists  uniform  positive constants  $\Lambda_3,\Lambda_4,\Lambda_5$,  depending  on the known data,  such that there holds
  \begin{equation}\label{e3.3}
 |G_{i}|\leq \Lambda_3,
 \end{equation}
\begin{equation}\label{e3.4}
\Lambda_4\leq \sum^{n}_{i=1}\frac{\partial G}{\partial \lambda_{i}}\leq \Lambda_5
\end{equation}
\begin{equation}\label{e3.5}
\Lambda_4\leq \sum^{n}_{i=1}\frac{\partial G}{\partial \lambda_{i}}\lambda^{2}_{i}\leq \Lambda_5
\end{equation}
where $\lambda_1,\cdots, \lambda_n$ are the eigenvalues of Hessian matrix $D^2 u$ at $x \in \Omega$.
\end{Corollary}
 Gathering the results obtained above we arrive at the following strucure conditions for the operator $\tilde{G}$.
\begin{Corollary}\label{c3.4}
 Assume that $\Omega$, $\tilde{\Omega}$ are bounded, uniformly convex domains with smooth boundary in $\mathbb{R}^{n}$ and  $\tilde{\Omega}\subset\subset B_{1}(0)$, $0<\alpha_{0}<1$, $u_{0}\in C^{2+\alpha_{0}}(\bar{\Omega})$
  which is   uniformly convex and satisfies $Du_{0}(\Omega)=\tilde{\Omega}$. If the strictly convex solution to (\ref{e2.17}) exists.
There   exists  uniform  positive constants  $\Lambda_3,\Lambda_4,\Lambda_5$,  depending  on the known data,  such that there holds
  \begin{equation}\label{e3.6}
 |\tilde{G}_{y_i}|\leq \Lambda_3,
 \end{equation}
\begin{equation}\label{e3.7}
\Lambda_4\leq \sum^{n}_{i=1}\frac{\partial \tilde{G}}{\partial \mu_{i}}\leq \Lambda_5
\end{equation}
\begin{equation}\label{e3.8}
\Lambda_4\leq \sum^{n}_{i=1}\frac{\partial \tilde{G}}{\partial \mu_{i}}\mu^{2}_{i}\leq \Lambda_5
\end{equation}
where $\mu_1,\cdots, \mu_n$ are the eigenvalues of Hessian matrix $D^2 \tilde{u}$ at $x \in \Omega$.
\end{Corollary}
We will need the following elementary lemma in order to obtain the uniformly oblique estimates later.  This follows from Lemma 3.4 in \cite{RS}.
\begin{lemma}\label{l3.30}
Assume that  $[A_{ij}]$ is semi-positive real symmetric matrix and $[B_{ij}]$, $[C_{ij}]$ are two real symmetric matrixes. Then
$$2A_{ij}B_{jk}C_{ki}\leq A_{ij}B_{ik}B_{jk}+A_{ij}C_{ik}C_{jk}.$$
\end{lemma}

According to the proof in \cite{JU}, we can verify the oblique boundary condition.
\begin{lemma}\label{l3.2}(See J. Urbas \cite{JU}.) Let $\nu=(\nu_{1},\nu_{2}, \cdots,\nu_{n})$ be the unit inward normal vector of $\partial\Omega$. If $u\in C^{2}(\bar{\Omega})$  with $D^{2}u\geq0$, then there holds $h_{p_{k}}(Du)\nu_{k}\geq0$.
\end{lemma}

In order to obtain the uniformly convex $C^{2}$ estimates,  we make use of the method to do the strict obliqueness estimates, a parabolic version of a result of J.Urbas \cite{JU} which was given in \cite{OK}. Returning to Lemma \ref{l3.2}, we get a uniform positive lower bound of the quantity $\inf_{\partial\Omega}h_{p_{k}}(Du)\nu_{k}$ which does not depend on $t$ under the structure conditions of $F$ in (\ref{e1.2.a})(\ref{e1.2.ab})(\ref{e1.2.1})(\ref{e3.1})(\ref{e3.2}).
\begin{rem}
Without loss of generality, in the following we set $C_{1},C_{2}, \cdots,$ to be constants depending only on the known data and  being independent of $t$.
\end{rem}
Here, we establish uniformly oblique estimates for smooth convex solutions of the mean curvature flow with the second boundary condition
. These will be important in the proof of both the long time existence and convergence result.

\begin{lemma}\label{llll3.4}
  Assume that $\Omega$, $\tilde{\Omega}$ are bounded, uniformly convex domains with smooth boundary in $\mathbb{R}^{n}$ and
   $\tilde{\Omega}\subset\subset B_{1}(0)$, $0<\alpha_{0}<1$, $u_{0}\in C^{2+\alpha_{0}}(\bar{\Omega})$ which is   uniformly convex and satisfies $Du_{0}(\Omega)=\tilde{\Omega}$. If $u$ is a strictly convex solution to (\ref{e1.12})-(\ref{e1.14}), then the strict obliqueness estimate
\begin{equation}\label{ee3.9}
\langle\beta, \nu\rangle\geq \frac{1}{C_1}>0
\end{equation}
holds on $\partial \Omega$ for some universal constant $C_1$, which depends only on  $u_0$, $\Omega,$ $\tilde{\Omega}$ and  is independent of $t$.
\end{lemma}
\begin{proof}
The proof of this is similar in the parts to the proof of \cite{JU} , but it is different from the elliptic operators and the barrier functions which are given in this paper. It follows from the maximum principle according to the struction conditions of the operator $G$, and is proved in the same way as \cite{RS}.
Define
$$\omega=\langle \beta,\nu\rangle+\tau h(Du),$$
where $\tau$ is a positive constant to be determined. Let $(x_{0},t_{0})\in \partial\Omega\times(0,T]$  such that
$$\langle \beta,\nu\rangle(x_0,t_{0})=h_{p_k}(Du(x_0,t_{0}))\nu_k(x_0)=\min_{\partial\Omega\times[0,T]}\langle \beta,\nu\rangle.$$
By rotation, we may assume that $\nu(x_0)=(0,\cdots,0,1)$. Applying the above assumptions and the boundary condition, we find that
$$\omega(x_0,t_{0})=\min_{\partial\Omega\times[0,T]} \omega=h_{p_{n}}(Du(x_{0},t_{0})).$$
By the smoothness of $\Omega$ and its convexity, we extend $\nu$ smoothly to a tubular neighborhood of $\partial\Omega$ such that in the matrix sense
\begin{equation}\label{e3.4}
  \left(\nu_{kl}\right):=\left(D_k\nu_l\right)\leq -\frac{1}{C_{2}}\operatorname{diag} (\underbrace {1,\cdots, 1}_{n-1},0),
\end{equation}
where $C_{2}$ is a positive constant. By Lemma \ref{l3.2}, we see that $h_{p_{n}}(Du(x_{0},t_{0}))\geq0$.

At the minimum point  $(x_0,t_{0})$  it yields
\begin{equation}\label{e3.11a}
 0=\omega_r=h_{p_np_k}u_{kr}+h_{p_k}\nu_{kr}+\tau h_{p_k}u_{kr}, \quad 1\leq r\leq n-1.
\end{equation}
We assume that the following key result
\begin{equation}\label{e3.6}
 \omega_n(x_0,t_{0})>-C_{3},
\end{equation}
holds which will be proved later, where $C_{3}$ is a positive constant depending only on $\Omega$ and $\tilde{\Omega}$. We observe that (\ref{e3.6}) can be rewritten as
\begin{equation}\label{e3.7}
h_{p_np_k}u_{kn}+h_{p_k}\nu_{kn}+ \tau h_{p_k}u_{kn}>-C_{3}.
\end{equation}
Multiplying $h_{p_n}$ on both sides of (\ref{e3.7}) and $h_{p_r}$ on both sides of (\ref{e3.11a}) respectively, and summing up together, one gets
\begin{equation*}
 \tau h_{p_k}h_{p_l}u_{kl}\geq -C_{3}h_{p_n}- h_{p_k}h_{p_l}\nu_{kl}- h_{p_k}h_{p_np_l}u_{kl}.
\end{equation*}
Combining (\ref{e3.4}) with
$$ 1\leq r\leq n-1,\quad h_{p_k}u_{kr}=\frac{\partial h(Du)}{\partial x_r}=0,\quad h_{p_k}u_{kn}=\frac{\partial h(Du)}{\partial x_n}\geq 0,\quad -h_{p_np_n}\geq 0,$$
we have
$$\tau h_{p_k}h_{p_l}u_{kl}\geq-C_{3}h_{p_n}+\frac{1}{C_{2}}|Dh|^2-\frac{1}{C_{2}}h^2_{p_n}
\geq-C_{4}h_{p_n}+\frac{1}{C_{4}}-\frac{1}{C_{4}}h^2_{p_n},$$
where $C_{4}=\max\{C_{2},C_{3}\}$.
Now, to obtain the estimate $\langle \beta,\nu\rangle$ we divide $-C_{4}h_{p_n}+\frac{1}{C_{4}}-\frac{1}{C_{4}}h^2_{p_n}$ into two cases at $(x_0,t_{0})$.

Case (i).  If
$$-C_{4}h_{p_n}+\frac{1}{C_{4}}-\frac{1}{C_{4}}h^2_{p_n}\leq \frac{1}{2C_{4}},$$
then
$$h_{p_k}(Du)\nu_{k}=h_{p_n}\geq \sqrt{\frac{1}{2}+\frac{C^{4}_{4}}{4}}-\frac{C^{2}_{4}}{2}.$$
It means that there is a uniform positive lower bound for $\underset{\partial\Omega}\min \langle \beta,\nu\rangle$.

Case (ii). If
$$-C_{4}h_{p_n}+\frac{1}{C_{4}}-\frac{1}{C_{4}}h^2_{p_n}> \frac{1}{2C_{4}},$$
then we know that there is a positive lower bound for $h_{p_k}h_{p_l}u_{kl}$.

Let $\tilde{u}$ be the Legendre transformation of $u$, then $\tilde{u}$ satisfies (\ref{e2.17}).
Here we emphasize that the structure conditions of $\tilde{G}$ as same as $G$ by Corollary \ref{c3.4}.
The unit inward normal vector of $\partial\Omega$ can be represented as $\nu=D\tilde{h}$. By the same token,
$\tilde{\nu}=Dh$, where $\tilde{\nu}=(\tilde{\nu}_{1}, \tilde{\nu}_{2},\cdots,\tilde{\nu}_{n})$ is the unit inward normal vector of $\partial\tilde{\Omega}$.

Let $\tilde{\beta}=(\tilde{\beta}^{1}, \cdots, \tilde{\beta}^{n})$ with $\tilde{\beta}^{k}:=\tilde{h}_{p_{k}}(D\tilde{u})$.
Using the representation as the works of \cite{HRY} and \cite{OK},
we also define
$$\tilde{\omega}=\langle\tilde{\beta}, \tilde{\nu}\rangle+\tilde{\tau} \tilde{h}(D\tilde{u}),$$
in which
$$\langle\tilde{\beta}, \tilde{\nu}\rangle=\langle\beta, \nu\rangle,$$
and $\tilde{\tau}$ is a  positive constant to be determined.
Denote $y_{0}=Du(x_{0},t_{0})$. Then $$\tilde{\omega}(y_{0},t_{0})=\omega(x_{0},t_{0})=\min_{\partial\tilde{\Omega}} \tilde{\omega}.$$ Using the same methods, under the assumption of
\begin{equation}\label{e3.9}
\tilde{\omega}_{n}(y_{0},t_{0})\geq -C_{5},
\end{equation}
we obtain the positive lower bounds of $\tilde{h}_{p_{k}}\tilde{h}_{p_{l}}\tilde{u}_{kl}$ or
$$h_{p_{k}}(Du)\nu_{k}=\tilde{h}_{p_{k}}(D\tilde{u})\tilde{\nu}_{k}=\tilde{h}_{p_{n}}\geq
\sqrt{\frac{1}{2}+\frac{C^{4}_{6}}{4}}-\frac{C^{2}_{6}}{2}.$$
On the other hand, one can easily check that
$$\tilde{h}_{p_{k}}\tilde{h}_{p_{l}}\tilde{u}_{kl}=\nu_{i}\nu_{j}u^{ij}.$$
Then by the positive lower bounds of $h_{p_{k}}h_{p_{l}}u_{kl}$ and $\tilde{h}_{p_{k}}\tilde{h}_{p_{l}}\tilde{u}_{kl}$, the desired conclusion can be obtained by
\begin{equation*}
\langle \beta,\nu\rangle=\sqrt{h_{p_k}h_{p_l}u_{kl}u^{ij}\nu_i\nu_j}.
\end{equation*}
For details of the proof of the above formula, see \cite{JU}. It remains to prove the key estimates (\ref{e3.6}) and (\ref{e3.9}).

At first we give the proof of (\ref{e3.6}). By (\ref{e2.16}), Corollary \ref{c3.3} and Lemma \ref{l3.30}, we have
\begin{equation*}\label{e3.10}
 \begin{aligned}
\mathcal{L}\omega=&G_{ij}u_{il}u_{jm}(h_{p_{k}p_{l}p_{m}}\nu_{k}+\tau h_{p_{l}p_{m}})\\
&+2G_{ij}h_{p_{k}p_{l}}u_{li}\nu_{kj}+G_{ij}h_{p_{k}}\nu_{kij}+G_{i}h_{p_{k}}\nu_{ki}\\
\leq& (h_{p_{k}p_{l}p_{m}}\nu_{k}+\tau h_{p_{l}p_{m}}+\delta_{lm})G_{ij}u_{il}u_{jm}+C_{7}\mathcal{T}_{G}+C_{8},
 \end{aligned}
\end{equation*}
where $\mathcal{L}=G_{ij}\partial_{ij}+G_{i}\partial_{i}-\partial_{t}$ and
$$2G_{ij}h_{p_{k}p_{l}}u_{li}\nu_{kj}\leq  G_{ij}u_{im}u_{mj}+C_{7}\mathcal{T}_{G}.$$
Since $D^{2}h\leq-\theta I$, we may choose $\tau$ large enough depending on the known data such that
$$(h_{p_{k}p_{l}p_{m}}\nu_{k}+\tau h_{p_{l}p_{m}}+\delta_{lm})<0.$$
Consequently, we deduce that
\begin{equation}\label{e3.11}
\mathcal{L}\omega\leq C_{9}\mathcal{T}_{G} \    \ in \   \ \Omega.
\end{equation}
by the convexity of $u$.

Denote a neighborhood of $x_{0}$ in $\Omega$ by
$$\Omega_{r}:=\Omega\cap B_{r}(x_{0}),$$
where $r$ is a positive constant such that $\nu$ is well defined in $\Omega_{r}$. In order to obtain the desired results, it suffices to consider the auxiliary function
$$\Phi(x,t)=\omega(x,t)-\omega(x_{0},t_{0})+\sigma\ln(1+k\tilde{h}(x))+A|x-x_{0}|^{2},$$
where $\sigma$, $k$ and $A$ are positive constants to be determined.
By  $\tilde{h}$ being the defining function of $\Omega$ and $G_{i}$ being bounded one can show that
\begin{equation}\label{e3.12}
 \begin{aligned}
  \mathcal{L}(\ln(1+k\tilde{h}))&=G_{ij}\left(\frac{k\tilde{h}_{ij}}{1+k\tilde{h}}
  -\frac{k\tilde{h}_{i}}{1+k\tilde{h}}\frac{k\tilde{h}_{j}}{1+k\tilde{h}}\right)
  +G_{i}\frac{k\tilde{h}_{i}}{1+k\tilde{h}}\\
  &\triangleq G_{ij}\frac{k\tilde{h}_{ij}}{1+k\tilde{h}}-G_{ij}\eta_{i}\eta_{j}+G_{i}\eta_{i}\\
  &\leq\left(-\frac{k\tilde{\theta}}{1+k\tilde{h}}+C_{10}-C_{11}|\eta-C_{12}I|^{2}\right)\mathcal{T}_{G}\\
  &\leq\left(-\frac{k\tilde{\theta}}{1+k\tilde{h}}+C_{10}\right)\mathcal{T}_{G},
  \end{aligned}
\end{equation}
where $\eta=\left(\frac{k\tilde{h}_{1}}{1+k\tilde{h}}, \frac{k\tilde{h}_{2}}{1+k\tilde{h}},\cdots, \frac{k\bar{h}_{n}}{1+k\tilde{h}}\right)$.
By taking $r$ to be small enough such that we have
\begin{equation}\label{e3.13}
 \begin{aligned}
0\leq \tilde{h}(x)&=\tilde{h}(x)-\tilde{h}(x_{0})\\
&\leq \sup_{\Omega_{r}}|D\tilde{h}||x-x_{0}|\\
&\leq r\sup_{\Omega}|D\tilde{h}|\leq \frac{\tilde{\theta}}{3C_{10}}.
 \end{aligned}
\end{equation}
By choosing $k=\frac{7C_{10}}{\tilde{\theta}}$ and applying (\ref{e3.13}) to (\ref{e3.12}) we obtain
\begin{equation}\label{e3.14}
  \mathcal{L}(\ln(1+k\tilde{h}))
  \leq -C_{10}\mathcal{T}_{G}.
\end{equation}
Combining (\ref{e3.11}) with (\ref{e3.14}), a direct computation yields
\begin{equation*}
\mathcal{L}(\Phi(x,t))\leq (C_{9}-\sigma C_{10}+2A+2AC_{13})\mathcal{T}_{G}.
\end{equation*}
On $\partial\Omega\times[0,T]$, it is clear that $\Phi(x,t)\geq0$.
Because $\omega$ is bounded, then it follows that we can choose $A$ large enough depending on the known data such that on $(\Omega\cap\partial B_{r}(x_{0}))\times[0,T]$,
\begin{equation*}
 \begin{aligned}
  \Phi(x,t)&=\omega(x,t)-\omega(x_{0},t_{0})+\sigma\ln(1+kh^{\ast})+Ar^{2}\\
  &\geq\omega(x,t)-\omega(x_{0},t_{0})+Ar^{2}\geq 0.
  \end{aligned}
\end{equation*}
The above argument implies that
\begin{equation}\label{e3.15aaa}
  \Phi(x,t)\geq 0, \,\,(x,t)\in(\partial\Omega_{r}\times[0,T])
\end{equation}
Furthermore, at $t=0$,
 \begin{equation}\label{e3.12aa}
 \begin{aligned}
  \triangle(\ln(1+k\tilde{h}))&=\sum^{n}_{i=1}\left(\frac{k\tilde{h}_{ii}}{1+k\tilde{h}}
  -\frac{k\tilde{h}_{i}}{1+k\tilde{h}}\frac{k\tilde{h}_{i}}{1+k\tilde{h}}\right)\\
  &\leq-\frac{kn\tilde{\theta}}{1+k\tilde{h}}\\
  &\leq-\frac{1}{\Lambda_{5}}\frac{kn\tilde{\theta}}{1+k\tilde{h}}\mathcal{T}_{G},
  \end{aligned}
\end{equation}
consequently,
\begin{equation*}
\triangle(\Phi(x,0))\leq \frac{1}{\Lambda_{5}}(\max_{x\in\bar{\Omega}}|\triangle\omega(x,0)|-\frac{21n}{10}\sigma C_{10}+2A)\mathcal{T}_{G}.
\end{equation*}
Let
$$\sigma=\frac{C_{9}+2A+2AC_{13}+\max_{x\in\bar{\Omega}}|\triangle\omega(x,0)|}{C_{10}},$$
then
\begin{equation}\label{e3.15aa}
\left\{ \begin{aligned}
   \triangle(\Phi(x,0))&\leq 0,\ \  &&x\in\Omega_{r},\\
   \Phi(x,0))&\geq 0,\ \  &&x\in\partial\Omega_{r},
                          \end{aligned} \right.
\end{equation}
\begin{equation}\label{e3.15}
   \mathcal{L}\Phi\leq 0,\,\,  x\in\Omega_{r}\times[0,T].
\end{equation}
According to the maximum principle in (\ref{e3.15aa}), we get that
\begin{equation}\label{e3.15aaaa}
\Phi(x,0)|_{\Omega_{r}}\geq 0.
\end{equation}
By making use of the maximum principle in (\ref{e3.15}),
it follows from (\ref{e3.15aaa}) and (\ref{e3.15aaaa}) that
$$\Phi|_{\Omega_{r}\times[0,T]}\geq \min_{\partial(\Omega_{r}\times[0,T])}\Phi\geq 0.$$
By the above inequality and $\Phi(x_0,t_0)=0$, we have $\partial_n\Phi(x_0,t_0)\geq 0$, which gives the desired estimate (\ref{e3.6}).

Finally, we are turning to the proof of (\ref{e3.9}). The proof of (\ref{e3.9}) is similar to that of (\ref{e3.6}). Define
$$\mathcal{\tilde{L}}=\tilde{G}_{ij}\partial_{ij}-\partial_{t}.$$
From Corollary \ref{c3.4}, we arrive at
\begin{equation*}
 \begin{aligned}
\mathcal{\tilde{L}}\tilde{\omega}=&\tilde{G}_{ij}\tilde{u}_{li}\tilde{u}_{mj}(\tilde{h}_{q_{k}q_{l}q_{m}}\tilde{\nu}_{k}+\tilde{\tau} \tilde{h}_{q_{l}q_{m}})+2\tilde{G}_{ij}\tilde{h}_{q_{k}q_{l}}\tilde{u}_{li}\tilde{\nu}_{kj}\\
&-\tilde{G}_{y_{k}}(\tilde{h}_{q_{k}q_{m}}\tilde{\nu}_{m}+\tilde{\tau} \tilde{h}_{q_{k}})+\tilde{G}_{ij}\tilde{h}_{q_{k}}\tilde{\nu}_{kij}\\
\leq&(\tilde{h}_{q_{k}q_{l}q_{m}}\tilde{\nu}_{k}+\tilde{\tau} \tilde{h}_{q_{l}q_{m}}+\delta_{lm})\tilde{G}_{ij}\tilde{u}_{il}\tilde{u}_{jm}+C_{14}\mathcal{T}_{\tilde{G}}+C_{15}(1+\tilde{\tau}),
 \end{aligned}
\end{equation*}
where
$$2\tilde{G}_{ij}\tilde{h}_{q_{k}q_{l}}\tilde{u}_{li}\tilde{\nu}_{kj}\leq \delta_{lm}\tilde{G}_{ij}\tilde{u}_{il}\tilde{u}_{jm}
+C_{14}\mathcal{T}_{\tilde{G}},$$
by Lemma \ref{l3.30}. Since $D^{2}\tilde{h}\leq-\tilde{\theta}I$, we only need to choose $\tilde{\tau}$ sufficiently large depending on the known data such that
$$\tilde{h}_{q_{k}q_{l}q_{m}}\tilde{\nu}_{k}+\tilde{\tau} \tilde{h}_{q_{l}q_{m}}+\delta_{lm}<0.$$
Therefore,
\begin{equation}\label{e3.16}
\mathcal{\tilde{L}}\tilde{\omega}\leq C_{16}\mathcal{T}_{\tilde{G}}
\end{equation}
by the convexity of $\tilde{u}$.

Denote a neighborhood of $y_{0}$ in $\tilde{\Omega}$ by
$$\tilde{\Omega}_{\rho}:=\tilde{\Omega}\cap B_{\rho}(y_{0}),$$
where $\rho$ is a positive constant such that $\tilde{\nu}$ is well defined in $\tilde{\Omega}_{\rho}$. In order to obtain the desired results,
we  consider the  auxiliary function
$$\Psi(y,t)=\tilde{\omega}(y,t)-\tilde{\omega}(y_{0},t_{0})+\tilde{k}h(y)+\tilde{A}|y-y_{0}|^{2},$$
where $\tilde{k}$ and $\tilde{A}$ are positive constants to be determined. It is easy to check that $\Psi(y)\geq0$ on $\partial\tilde{\Omega}\times[0,T]$. Now that $\tilde{\omega}$ is bounded, it follows that we can choose $\tilde{A}$ large enough depending on the known data such that on $(\tilde{\Omega}\cap\partial B_{\rho}(y_{0}))\times[0,T]$,
\begin{equation*}
  \Psi(y,t)=\tilde{\omega}(y.t)-\tilde{\omega}(y_{0},t_{0})+\tilde{k}h(y)+\tilde{A}\rho^{2}\geq \tilde{\omega}(y,t)-\tilde{\omega}(y_{0},t_{0})
  +\tilde{A}\rho^{2}\geq 0.
\end{equation*}
It follows from (\ref{e3.16}) and $D^{2}h\leq-\theta I$ that
$$\mathcal{\tilde{L}}\Psi\leq(C_{16}-\tilde{k}\theta+2\tilde{A})\mathcal{T}_{\tilde{G}},$$
$$\triangle\Psi(y,0)\leq\frac{1}{\Lambda_{5}}(\max_{y\in\bar{\tilde{\Omega}}_{\rho}}|\triangle\tilde{\omega}(y.0)|-n\tilde{k}\theta+2n\tilde{A})\mathcal{T}_{\tilde{G}},$$
Let $$\tilde{k}=\frac{2n\tilde{A}+C_{16}+\max_{y\in\bar{\tilde{\Omega}}_{\rho}}|\triangle\tilde{\omega}(y.0)|}{\theta}$$
Consequently,
\begin{equation*}
\left\{ \begin{aligned}
   \mathcal{\tilde{L}}\Psi&\leq 0,\ \  &&(y,t)\in\tilde{\Omega}_{\rho}\times[0,T],\\
   \Psi&\geq 0,\ \  &&y\in\partial(\tilde{\Omega}_{\rho}\times[0,T]).
                          \end{aligned} \right.
\end{equation*}
The rest of the proof of (\ref{e3.9}) is the same as (\ref{e3.6}). Thus the proof of (\ref{ee3.9}) is completed.
\end{proof}

We now proceed to carry out the global $C^2$ estimate. Note that the results of Lemma \ref{l3.1} involve an application to the upper $C^{2}$ estimates of the mean curvature flow with second boundary condition.

\begin{lemma}\label{lem3.10}
 Assume that $\Omega$, $\tilde{\Omega}$ are bounded, uniformly convex domains with smooth boundary in $\mathbb{R}^{n}$ and $\tilde{\Omega}\subset\subset B_{1}(0)$, $0<\alpha_{0}<1$, $u_{0}\in C^{2+\alpha_{0}}(\bar{\Omega})$ which is   uniformly convex and satisfies $Du_{0}(\Omega)=\tilde{\Omega}$. If $u$ is a strictly convex solution to (\ref{e1.12})-(\ref{e1.14}), then there exists a positive constant $C_{17}$ depending only on $u_0$, $\Omega$, $\tilde{\Omega}$, such that
\begin{equation}\label{eq3.15}
D^2 u(x,t) \leq C_{17} I_n,\ \ (x,t)\in\bar\Omega_T,
\end{equation}
where $I_n$ is the $n\times n$ identity matrix.
\end{lemma}
\begin{proof}
By the proof of Lemma \ref{l3.3}, one can show that $$\min_{p\in \tilde{\Omega}}\frac{1}{\sqrt{1-|p|^{2}}}\min_{\bar{\Omega}}G(Du_{0},D^{2}u_{0})\leq  \sum_{i=1}^n\kappa_{i}\leq \max_{p\in \tilde{\Omega}}\frac{1}{\sqrt{1-|p|^{2}}}\max_{\bar{\Omega}}G(Du_{0},D^{2}u_{0}).$$
By the relationship in (\ref{e2.71}), we get that
$$\sum_{i=1}^n\kappa_{i}=\sum_{i=1}^n\frac{1}{v}b^{ik}D_{kl}ub^{li}.$$
Then, using \begin{equation*}
b^{ij}=\delta_{ij}-\frac{D_{i}uD_{j}u}{v(1+v)}.
\end{equation*}
and the second boundary condition,   we obtain the desired bound.
\end{proof}
In the following,  we describe the positive lower bound of $D^{2}u$.  Starting from the evolution equation (\ref{e1.11a}), we will exhibit the evolution equations for the metric $g$, the second fundamental form $A=h_{ij}$, the mean curvature $H$, seeing the reference in \cite{KG}.
\begin{equation}\label{e3.23}
\frac{\partial g_{ij}}{\partial t}=2Hh_{ij},
\end{equation}
\begin{equation}\label{e3.25}
\frac{\partial h_{ij}}{\partial t}=\nabla_{i}\nabla_{j}H+Hh_{il}g^{lk}h_{kj},
\end{equation}
\begin{equation}\label{e3.27}
\triangle_{M} h_{ij}=\nabla_{i}\nabla_{j}H+h_{ij}|A|^{2}-Hh_{il}g^{lk}h_{kj},
\end{equation}
\begin{equation}\label{e3.29}
(\frac{\partial}{\partial t}-\triangle_{M})h_{ij}=2Hh_{il}g^{lk}h_{kj}-h_{ij}|A|^{2},
\end{equation}
\begin{equation}\label{e3.30}
(\frac{\partial}{\partial t}-\triangle_{M} )H=-|A|^{2}H,
\end{equation}
where $\nabla_{i}, \nabla_{j},\cdots,$ denote the covariant derivative on the manifold $M\triangleq \{(x,u(x,t)|x\in\Omega\}\subset \mathbb{R}^{n,1}$ , $\triangle_{M}$ is the Laplace-Beltrami operator on $M$, $[h^{ij}]=[h_{ij}]^{-1}$ and $|A|^{2}=h_{il}g^{lk}h_{kj}g^{ij}.$

\begin{lemma}\label{lem3.11}
 Assume that $\Omega$, $\tilde{\Omega}$ are bounded, uniformly convex domains with smooth boundary in $\mathbb{R}^{n}$ and $\tilde{\Omega}\subset\subset B_{1}(0)$, $0<\alpha_{0}<1$, $u_{0}\in C^{2+\alpha_{0}}(\bar{\Omega})$ which is   uniformly convex and satisfies $Du_{0}(\Omega)=\tilde{\Omega}$. If $u$ is a strictly convex solution to (\ref{e1.12})-(\ref{e1.14}) and  $\epsilon_{0}$ is some constant satisfying
\begin{equation*}\label{e3.32}
h_{ij}\geq \epsilon_{0}Hg_{ij},\ \ (x,t)\in\partial\Omega_T.
\end{equation*}
Then there holds
\begin{equation*}\label{e3.33}
h_{ij}\geq \epsilon_{0}H g_{ij},\ \ (x,t)\in\Omega_T.
\end{equation*}
\end{lemma}
\begin{proof}
On the  strictly spacelike graph $M$, let $u^{k}$ be a vector field with $$u^{k}=\frac{2}{H}g^{kl}\nabla_{l}H$$
and $M_{ij}$,  $N_{ij}$ be symmetric tensors as follows
$$M_{ij}=\frac{h_{ij}}{H}-\epsilon_{0}g_{ij},\quad\,\,N_{ij}=2Hh_{il}g^{lk}h_{kj}-2\epsilon_{0}Hh_{ij}.$$
As same as  the computation in Theorem 4.3 from \cite{GH},  $N_{ij}$  is a polynomial in $M_{ij}$  and satisfy a null-eigenvector condition.
It follows from (\ref{e3.23}), (\ref{e3.29}), (\ref{e3.30}) that one can gives the evolution equation
$$(\frac{\partial}{\partial t}-\triangle_{M} )M_{ij}=u^{k}\nabla_{k}M_{ij}+N_{ij}.$$
Then it yields the desired results by Hamilton's maximum principle \cite{RSH}.
\end{proof}

\begin{Corollary}\label{c3.12}
 Assume that $\Omega$, $\tilde{\Omega}$ are bounded, uniformly convex domains with smooth boundary in $\mathbb{R}^{n}$ and $\tilde{\Omega}\subset\subset B_{1}(0)$, $0<\alpha_{0}<1$, $u_{0}\in C^{2+\alpha_{0}}(\bar{\Omega})$ which is   uniformly convex and satisfies $Du_{0}(\Omega)=\tilde{\Omega}$. If $u$ is a strictly convex solution to (\ref{e1.12})-(\ref{e1.14}),
then there exists a positive constant $C_{18}$ depending only on $u_0$, $\Omega$, $\tilde{\Omega}$, such that
\begin{equation}\label{eq3.16}
D^2 u(x,t) \geq C_{18}\min_{\partial\Omega_T}|D^2 u| I_n,\ \ (x,t)\in\bar\Omega_T,
\end{equation}
where $I_n$ is the $n\times n$ identity matrix.
\end{Corollary}
\begin{proof}
By $\tilde{\Omega}\subset\subset B_{1}(0)$ and the second boundary consition  we see that
\begin{equation*}
h_{ij}\geq \epsilon_{0}Hg_{ij},\ \ (x,t)\in\partial\Omega_T.
\end{equation*}
where $\epsilon_{0}=C\min_{\partial\Omega_T}|D^2 u|$ for some universal consttant $C$ depending only on the known data but independence of $T$.
Using the previous lemma we derive
\begin{equation*}
h_{ij}\geq \epsilon_{0}H g_{ij},\ \ (x,t)\in\Omega_T.
\end{equation*}
Then the conclusion follows by the boundedness of $H$  and  $\tilde{\Omega}\subset\subset B_{1}(0)$.
\end{proof}
To obtain  the positive lower bound of $D^{2}u$ on $\partial\Omega_T$, we consider of the Legendre transformation of $u$.
As before, we see that this can in fact be written
\begin{equation*}
\frac{\partial \tilde{u}}{\partial y_{i}}=x_{i},\,\,\frac{\partial ^{2}\tilde{u}}{\partial y_{i}\partial y_{j}}=u^{ij}(x)
\end{equation*}
where $[u^{ij}]=[D^{2}u]^{-1}$. Then $\tilde{u}$ satisfies (\ref{e2.17}).

Recall that $\tilde{\beta}=(\tilde{\beta}^{1}, \cdots, \tilde{\beta}^{n})$ with $\tilde{\beta}^{k}:=\tilde{h}_{p_{k}}(D\tilde{u})$ and $\tilde{\nu}=(\tilde{\nu}_{1}, \tilde{\nu}_{2},\cdots,\tilde{\nu}_{n})$ is the unit inward normal vector of $\partial\tilde{\Omega}$.
In the following we give the arguments as in \cite{JU}, one can see there for more details.
For any tangential direction $\tilde \varsigma$, we have
\begin{equation}\label{e3.35}
   \tilde{u}_{\tilde\beta \tilde\varsigma}=\tilde h_{p_k}(D\tilde u)\tilde u_{k\tilde\varsigma}=0.
\end{equation}
Then the second order derivative of $\tilde u$ on the boundary is also controlled by $u_{\tilde\beta \tilde\varsigma}$, $u_{\tilde\beta \tilde\beta}$ and $u_{\tilde\varsigma\tilde\varsigma}$. At $\tilde x\in \partial\tilde\Omega$, any unit vector $\tilde\xi$ can be written in terms of a tangential component $\tilde\varsigma(\tilde\xi)$ and a component in the direction $\tilde\beta$ by
$$\tilde\xi=\tilde\varsigma(\tilde\xi)+\frac{\langle \tilde\nu,\tilde\xi\rangle}{\langle\tilde\beta,\tilde\nu\rangle}\tilde\beta,$$
where
$$\tilde\varsigma(\tilde\xi):=\tilde\xi-\langle \tilde\nu,\tilde\xi\rangle \tilde\nu-\frac{\langle \tilde\nu,\tilde\xi\rangle}{\langle\tilde\beta,\tilde\nu\rangle}\tilde\beta^T,$$
and
$$\tilde\beta^T:=\tilde\beta-\langle \tilde\beta,\tilde\nu\rangle \tilde\nu.$$
We observe that $\langle\tilde\beta,\tilde\nu\rangle=\langle\beta,\nu\rangle$.
By the uniformly obliqueness estimate (\ref{ee3.9}), we have
\begin{equation}\label{e3.36}
\begin{aligned}
|\tilde{\varsigma}(\tilde{\xi})|^{2}&=1-\left(1-\frac{|\tilde{\beta}^{T}|^{2}}{\langle\tilde{\beta},\tilde{\nu}\rangle^{2}}\right)
\langle\tilde{\nu},\tilde{\xi}\rangle^{2}
-2\langle\tilde{\nu},\tilde{\xi}\rangle\frac{\langle\tilde{\beta}^{T},\tilde{\xi}\rangle}{\langle\tilde{\beta},\tilde{\nu}\rangle}\\
&\leq 1+C_{19}\langle\tilde{\nu},\tilde{\xi}\rangle^{2}-2\langle\tilde{\nu},\tilde{\xi}\rangle\frac{\langle\tilde{\beta}^{T},\tilde{\xi}\rangle}{\langle\tilde{\beta},\tilde{\nu}\rangle}\\
&\leq C_{20}.
\end{aligned}
\end{equation}
Denote $\varsigma:=\frac{\varsigma(\xi)}{|\varsigma(\xi)|}$, then by (\ref{ee3.9}), (\ref{e3.35}) and (\ref{e3.36}) we arrive at
\begin{equation}\label{e3.37}
\begin{aligned}
\tilde{u}_{\tilde{\xi}\tilde{\xi}}&=|\tilde{\varsigma}(\tilde{\xi})|^{2}
\tilde{u}_{\tilde{\varsigma}\tilde{\varsigma}}+2|\tilde{\varsigma}(\tilde{\xi})|\frac{\langle\tilde{\nu},\tilde{\xi}\rangle}{\langle\tilde{\beta},\tilde{\nu}\rangle}\tilde{u}_{\tilde{\beta}\tilde{\varsigma}}+
\frac{\langle\nu,\xi\rangle^{2}}{\langle\beta,\nu\rangle^{2}}
\tilde{u}_{\tilde{\beta}\tilde{\beta}}\\
&=|\tilde{\varsigma}(\tilde{\xi})|^{2}\tilde{u}_{\tilde{\varsigma}\tilde{\varsigma}}+\frac{\langle\tilde{\nu},\tilde{\xi}\rangle^{2}}{\langle\tilde{\beta},\tilde{\nu}\rangle^{2}}
\tilde{u}_{\tilde{\beta}\tilde{\beta}}\\
&\leq C_{21}(\tilde{u}_{\tilde{\varsigma}\tilde{\varsigma}}+\tilde{u}_{\tilde{\beta}\tilde{\beta}}),
\end{aligned}
\end{equation}
Therefore, we also only need to estimate $\tilde{u}_{\tilde\beta\tilde\beta}$ and $\tilde{u}_{\tilde\varsigma\tilde\varsigma}$ respectively.

Further we have
\begin{lemma}\label{lem3.13a}
If $\tilde{u}$ is a strictly convex solution of (\ref{e2.17}), then there exists a positive constant $C_{22}$ depending only on $u_0$, $\Omega$, $\tilde{\Omega}$, such that
\begin{equation}\label{e3.42}
   \max_{\partial\Omega_T}\tilde{u}_{\tilde{\beta}\tilde{\beta}} \leq C_{22}.
\end{equation}
\end{lemma}
\begin{proof}
Without loss of generality, let $\tilde{x}_0\in\partial\tilde{\Omega}$, $t_0\in (0,T]$ satisfy $\tilde{u}_{\tilde{\beta}\tilde{\beta}}(\tilde{x}_0,t_0)=\max_{\partial\Omega_T}\tilde{u}_{\tilde{\beta}\tilde{\beta}}$.
To estimate the upper bound of $\tilde{u}_{\tilde{\beta}\tilde{\beta}}$,
we consider the barrier function
$$\tilde\Psi:=-\tilde{h}(D\tilde{u})+C_0 h.$$
For any $y\in \partial\tilde{\Omega}$, $t\in[0,T]$, $D\tilde{u}(y,t)\in \partial\Omega$, then $\tilde{h}(D\tilde{u})=0$. It is clear that $h=0$ on $\partial\tilde{\Omega}$.  First we have
$$\mathcal{\tilde{L}}(C_0 h)=C_0 \tilde{G}_{ij}h_{ij}\leq C_0\left(-\theta\sum_{i=1}^{n}\tilde{G}_{ii}\right).$$
Using the equations (\ref{e2.17}), a direct computation shows that
\begin{equation}\label{e3.39}
\begin{aligned}
\mathcal{\tilde{L}}\left(-\tilde{h}(D\tilde{u})\right)&=\tilde{G}_{ij}\left(-\tilde{h}_{\tilde{p}_{k}\tilde{p}_{l}}
\partial_{ki}\tilde{u}\partial_{lj}\tilde{u}\right)-\tilde{h}_{\tilde{p}_{k}}\tilde{G}_{y_{k}}\\
&\leq C_{23}\sum_{i=1}^{n}\tilde{G}_{ii},
\end{aligned}
\end{equation}
where we had used the estimates (\ref{e3.3})-(\ref{e3.5}) in Corollary \ref{c3.4}.
Therefore, we obtain
$$\tilde L\tilde\Psi(y)\leq \left(C_{23}-C_0\theta\right)\sum_{i=1}^n \tilde G_{ii}. $$
It is clear that $\tilde\Psi=0$ on $\partial\tilde\Omega\times [0,T]$. At $t=0$,
$$-\triangle\tilde{\Psi}=\triangle\tilde{h}(D\tilde{u}_{0})-\triangle C_{0}h\geq \triangle\tilde{h}(D\tilde{u}_{0})+C_{0}n\theta,$$
where we had used the properties of the defining function $h$ respect to $\tilde{\Omega}.$  Let
$$C_{0}=\frac{C_{23}}{\theta}+\frac{1}{n\theta}\sup_{\tilde{\Omega}}|\triangle\tilde{h}(D\tilde{u}_{0})|,$$
then $\mathcal{\tilde{L}}\tilde\Psi(y)\leq 0$  and $-\triangle\tilde{\Psi}\mid_{t=0}\geq 0.$ We see that at $t=0$,
\begin{equation*}
\left\{ \begin{aligned}
   -\triangle\tilde\Psi&\geq 0,\quad &&y\in\tilde\Omega,\\
   \tilde\Psi&= 0 ,\quad\  &&y\in\partial\tilde\Omega.
\end{aligned} \right.
\end{equation*}
By the maximum principle for elliptic partial differential equation, we arrive at
$$\tilde\Psi\mid_{t=0}\geq 0.$$
It follows from the above results that
\begin{equation}\label{e3.40}
\left\{ \begin{aligned}
   \mathcal{\tilde{L}}\tilde\Psi&\leq 0,\quad &&(y,t)\in\tilde\Omega\times(0,T],\\
   \tilde\Psi&\geq 0 ,\quad\  &&(y,t)\in(\partial\tilde\Omega\times(0,T])\cup(\tilde\Omega\times\{t=0\}).
\end{aligned} \right.
\end{equation}
Applying the maximum principle for parabolic partial differential equation, we get
$$\tilde\Psi(y,t)\geq 0,\quad\quad (y,t)\in \tilde{\Omega}\times[0,T].$$
Combining it with $\tilde\Psi(\tilde{x}_0,t_0)=0$ we obtain $\tilde\Psi_{\tilde{\beta}}(\tilde{x}_0,t_0)\geq 0$, which implies
$$\frac{\partial \tilde{h}}{\partial \tilde{\beta}}(D\tilde{u}(\tilde{x}_0,t_0))\leq C_0.$$
On the other hand, we see that at $(\tilde{x}_0,t_0)$,
$$\frac{\partial \tilde{h}}{\partial \tilde{\beta}}=\langle D\tilde{h}(D\tilde{u}),\tilde{\beta}\rangle=\frac{\partial \tilde{h}}{\partial p_k}\tilde{u}_{kl}\tilde{\beta}^l=\tilde{\beta}^k\tilde{u}_{kl}\tilde{\beta}^l=\tilde{u}_{\tilde{\beta}\tilde{\beta}}.$$
Therefore, letting $C_{22}= C_{0}$ we have
$$\tilde{u}_{\tilde{\beta}\tilde{\beta}}=\frac{\partial \tilde{h}}{\partial \tilde{\beta}}\leq C_{22}.$$
\end{proof}
\begin{lemma}\label{lem3.13}
Let $\tilde{G}=\tilde{G}(y,D^{2}\tilde{u})$,
$\tilde{G}_{ij,kl}=\frac{\partial^{2}\tilde{G}}{\partial\tilde{u}_{ij}\partial\tilde{u}_{kl}}$
and $\tilde{G}_{ij}=\frac{\partial\tilde{G}}{\partial\tilde{u}_{ij}}.$  If $\tilde{u}$ is a strictly convex solution of (\ref{e2.17}), then there exists a positive constant $C_{24}$ depending only on $u_0$, $\Omega$, $\tilde{\Omega}$, such that
\begin{equation}\label{e3.42a}
  -\tilde{G}_{ij,kl}\tilde{\eta}_{ij}\tilde{\eta}_{kl}+2\langle D_{y}\tilde{G}_{ij},\tilde{\xi}\rangle\tilde{\eta}_{ij}\geq -C_{24}\sum_{i=1}^{n}\tilde{G}_{ii}
\end{equation}
where $\eta_{ij}$ be any  symmetric tensor and $\tilde{\xi}$ be any unit vector.
\end{lemma}
\begin{proof}
Denote $s_{ij}=\delta_{ij}+\frac{y_{i}y_{j}}{1-|y|^{2}}$ and $[\tilde{u}_{ij}]=[u_{ij}]^{-1}.$  Recalling that $$\tilde{G}(y, D^{2}\tilde{u})=-s_{ij}u_{ij}.$$
By the rotation of the coordinate system for any fixed point $y\in \tilde{\Omega}$, we assume that
$$[\tilde{u}_{ij}]=\mathrm{diag}(\tilde{u}_{11},\tilde{u}_{22},\cdots\tilde{u}_{nn}).$$
We observe that
$$\tilde{s}_{ii}\triangleq \langle D_{y}\tilde{G}_{ii},\tilde{\xi}\rangle\tilde{u}^{2}_{ii}$$
depend only on $y$ and $\tilde{\xi}$, so that it's bounded.
Then there holds
\begin{equation*}
\begin{aligned}
 -\tilde{G}_{ij,kl}\tilde{\eta}_{ij}\tilde{\eta}_{kl}+2\langle D_{y}\tilde{G}_{ij},\tilde{\xi}\rangle\tilde{\eta}_{ij}
&=s_{ii}\frac{2}{\tilde{u}^{3}_{ii}}\tilde{\eta}^{2}_{ii}+\tilde{s}_{ii}\frac{1}{\tilde{u}^{2}_{ii}}\tilde{\eta}_{ii}\\
&\geq \frac{2}{\tilde{u}^{3}_{ii}}\tilde{\eta}^{2}_{ii}-C_{25}\frac{1}{\tilde{u}^{2}_{ii}}|\tilde{\eta}_{ii}|\\
&\geq \frac{2}{\tilde{u}^{3}_{ii}}\tilde{\eta}^{2}_{ii}-\frac{2}{\tilde{u}^{3}_{ii}}\tilde{\eta}^{2}_{ii}-
C^{2}_{25}\frac{2}{\tilde{u}_{ii}}\\
&\geq -C^{2}_{25}\frac{2}{\tilde{u}_{ii}}\\
&=-C^{2}_{25}2u_{ii}\\
&\geq -2C^{2}_{25}C_{17},
\end{aligned}
\end{equation*}
where we had used Lemma \ref{lem3.10}. It yields the desired result by setting $$C_{24}=\frac{2}{\Lambda_{4}}C^{2}_{25}C_{17}.$$\end{proof}

Next, we estimate the double tangential derivative of $\tilde{u}$.
\begin{lemma}\label{lem3.14}
If $\tilde{u}$ is a strictly convex solution of (\ref{e2.17}), then there exists a positive constant $C_{26}$ depending only on $u_0$, $\Omega$, $\tilde{\Omega}$, such that
\begin{equation*}\label{e4.43}
   \max_{\partial\tilde{\Omega}\times[0,T]}\max_{|\tilde\varsigma|=1, \langle\tilde\varsigma,\tilde\nu\rangle=0} \tilde{u}_{\tilde\varsigma\tilde\varsigma} \leq C_{26}.
\end{equation*}
\end{lemma}
\begin{proof}
Without loss of generality, we assume that $\tilde{x}_{0}\in\partial\Omega$, $t_{0}\in(0,T]$, $e_{n}$ is the unit inward normal vector of $\partial\tilde{\Omega}$ at $\tilde{x}_0$ and $e_{1}$ is the tangential veator of $\partial\tilde{\Omega}$ at $\tilde{x}_0$ respectively, such that
$$ \max_{\partial\tilde{\Omega}\times[0,T]}\max_{|\tilde\varsigma|=1, \langle\tilde\varsigma,\tilde\nu\rangle=0} \tilde{u}_{\tilde\varsigma\tilde\varsigma}=\tilde{u}_{11}(\tilde{x}_0,t_0)=:\mathcal{ M}.$$
For any $y\in \partial\tilde{\Omega}$, $t\in[0,T]$, it follows from the proof of (\ref{e3.36}) and (\ref{e3.37}) that
\begin{equation}\label{e4.45}
\begin{aligned}
\tilde u_{\tilde\xi\tilde\xi}&=|\tilde\varsigma(\tilde\xi)|^2\tilde u_{\tilde\varsigma\tilde\varsigma}+ \frac{\langle \tilde\nu,\tilde\xi\rangle^2}{\langle \tilde\beta,\tilde\nu\rangle^2}\tilde u_{\tilde\beta\tilde\beta}\\
          &\leq \left(1+C_{27}\langle \tilde\nu,\tilde\xi\rangle^2-2\langle \tilde\nu,\tilde\xi\rangle \frac{\langle \tilde\beta^T,\tilde\xi\rangle}{\langle \tilde\beta,\tilde\nu\rangle}\right) \mathcal{M}
                 + \frac{\langle \tilde\nu,\tilde\xi\rangle^2}{\langle \tilde\beta,\tilde\nu\rangle^2}\tilde u_{\tilde\beta\tilde\beta}.
\end{aligned}
\end{equation}
Let us skip therefore to the case $\mathcal{M}\geq 1$. Thus by (\ref{ee3.9}), (\ref{e3.42}) and (\ref{e4.45}) we deduce that
\begin{equation*}\label{eq3.9a}
  \frac{\tilde u_{\tilde\xi\tilde\xi}}{\tilde M}+2\langle \tilde\nu,\tilde\xi\rangle \frac{\langle \tilde\beta^T,\tilde\xi\rangle}{\langle \tilde\beta,\tilde\nu\rangle}
             \leq 1+C_{28}\langle \tilde\nu,\tilde\xi\rangle^2.
\end{equation*}
Let $\tilde\xi=e_1$, then
\begin{equation*}\label{eq3.10a}
  \frac{\tilde u_{11}}{\tilde M}+2\langle \tilde\nu,e_1\rangle \frac{\langle \tilde\beta^T,e_1\rangle}{\langle \tilde\beta,\tilde\nu\rangle}
             \leq 1+C_{28}\langle \tilde\nu,e_1\rangle^2.
\end{equation*}
We see that the function
\begin{equation*}\label{eq3.11a}
\tilde w:=A|y-\tilde x_0|^2-\frac{\tilde u_{11}}{\tilde M}-2\langle \tilde\nu,e_1\rangle \frac{\langle \tilde\beta^T,e_1\rangle}{\langle \tilde\beta,\tilde\nu\rangle}+C_{28}\langle \tilde\nu,e_1\rangle^2+1
\end{equation*}
satisfies
$$\tilde w|_{\partial\tilde\Omega\times[0,T]}\geq 0,\quad  \tilde w(\tilde x_0,t_0)=0.$$
Denote a neighborhood of $\tilde{x}_0$ in $\tilde\Omega$ by
$$\tilde\Omega_{r}:=\tilde\Omega\cap B_{r}(\tilde{x}_0),$$
where $r$ is a positive constant such that $\tilde\nu$ is well defined in $\tilde\Omega_{r}$.
Let us consider
$$-2\langle \tilde\nu,e_1\rangle \frac{\langle \tilde\beta^T,e_1\rangle}{\langle \tilde\beta,\tilde\nu\rangle}+C_{28}\langle \tilde\nu,e_1\rangle^2+1$$
as a known function depending on $y$ and $D\tilde u$. Then by the proof of (\ref{e3.39}), we also obtain
\begin{equation*}
\left|\mathcal{\tilde L}\left(-2\langle \tilde\nu,e_1\rangle \frac{\langle \tilde\beta^T,e_1\rangle}{\langle \tilde\beta,\tilde\nu\rangle}+C_{28}\langle \tilde\nu,e_1\rangle^2+1\right)\right|\leq C_{29}\sum_{i=1}^n \tilde{G}_{ii}.
\end{equation*}
By the equations (\ref{e2.17}), we get that
\begin{equation*}
\mathcal{\tilde{L}}\tilde{u}_{11}=-\tilde{G}_{ij,kl}\tilde{u}_{ij1}\tilde{u}_{kl1}+2\tilde{G}_{y_{1},ij}\tilde{u}_{ij1}.
\end{equation*}
It follows from (\ref{e3.42a}) in Lemma \ref{lem3.13} that
\begin{equation*}
\mathcal{\tilde{L}}\tilde{u}_{11}\geq-C_{24}\sum_{i=1}^{n}\tilde{G}_{ii}.
\end{equation*}
We set
$$\tilde\Upsilon:=\tilde w+C_0 {h}.$$
Furthermore, by(\ref{ee3.9}), (\ref{eq3.16}), (\ref{e3.37}) and (\ref{e3.42}),  we can choose the constant $A$ large enough such that
$$\tilde w|_{(\tilde\Omega \cap \partial B_{r}(\tilde x_0))\times[0,T]} \geq 0 $$
As in the proof of (\ref{e3.40}), letting $C_0>>A$ we get that
$$\mathcal{\tilde{L}}\tilde\Upsilon\leq 0,\quad (y,t)\in \tilde\Omega_{r}\times(0,T)$$
and
$$\tilde\Upsilon|_{t=0}\geq 0.$$
Putting the above arguments togather, we arrive at
\begin{equation*}\label{e3.47}
\left\{ \begin{aligned}
   \mathcal{\tilde{L}}\tilde\Upsilon&\leq 0,\quad &&(y,t)\in\tilde\Omega_{r}\times(0,T],\\
   \tilde\Upsilon&\geq 0 ,\quad\  &&(y,t)\in(\partial\tilde\Omega_{r}\times(0,T])\cup(\tilde\Omega_{r}\times\{t=0\}).
\end{aligned} \right.
\end{equation*}
A standard barrier argument makes conclusion of
$$\tilde\Upsilon_{\tilde\beta}(\tilde x_0,t_0)\geq0.$$
Therefore,
\begin{equation}\label{eq3.12a}
 \tilde u_{11\tilde\beta}(\tilde x_0)\leq C_{30}\mathcal{ M}.
\end{equation}
On the other hand, differentiating $\tilde h(D\tilde u)$ twice in the direction $e_1$ at $(\tilde x_0,t_0)$, we have
$$\tilde h_{p_k}\tilde u_{k11}+\tilde h_{p_kp_l}\tilde u_{k1}\tilde u_{l1}=0.$$
The concavity of $\tilde h$ yields that
$$\tilde h_{p_k}\tilde u_{k11}=-\tilde h_{p_kp_l}\tilde u_{k1}\tilde u_{l1}\geq \tilde\theta \mathcal{M}^2.$$
Combining it with $\tilde h_{p_k}\tilde u_{k11}=\tilde u_{11\tilde\beta}$, and using (\ref{eq3.12a}) we obtain
$$\tilde\theta  \mathcal{M}^2\leq C_{30}\mathcal{M}.$$
Then we get the upper bound of $\mathcal{M}=\tilde u_{11}(\tilde x_0,t_0)$ and thus the desired result follows.
\end{proof}

By Lemma \ref{lem3.13a}, Lemma \ref{lem3.14} and (\ref{e3.37}), we obtain the $C^2$ a-priori estimate of $\tilde{u}$ on the boundary.
\begin{lemma}\label{lem3.4a}
If $\tilde{u}$ is a strictly convex solution of (\ref{e2.17}), then there exists a positive constant $C_{31}$ depending only on $u_0$, $\Omega$, $\tilde{\Omega}$, such that
\begin{equation*}\label{eq3.13a}
\max_{\partial\tilde\Omega_T}|D^2\tilde u| \leq C_{31}.
\end{equation*}
\end{lemma}

By Corollary \ref{c3.12} and Lemma \ref{lem3.4a}, we can see that
\begin{lemma}\label{lem3.5a}
If $\tilde{u}$ is a strictly convex solution of (\ref{e2.17}), then there exists a positive constant $C_{32}$ depending only on $u_0$, $\Omega$, $\tilde{\Omega}$, such that
\begin{equation*}\label{eq3.14a}
\max_{\bar{\tilde\Omega}_T}|D^2\tilde u| \leq C_{32}.
\end{equation*}
\end{lemma}

By Lemma \ref{lem3.10} and Lemma \ref{lem3.5a}, we conclude that
\begin{lemma}\label{lem3.6}
Assume that $\Omega$, $\tilde{\Omega}$ are bounded, uniformly convex domains with smooth boundary in $\mathbb{R}^{n}$ and $\tilde{\Omega}\subset\subset B_{1}(0)$, $0<\alpha_{0}<1$, $u_{0}\in C^{2+\alpha_{0}}(\bar{\Omega})$ which is   uniformly convex and satisfies $Du_{0}(\Omega)=\tilde{\Omega}$. If $u$ is a strictly convex solution to (\ref{e1.12})-(\ref{e1.14}), then there exists a positive constant $C_{33}$ depending only on $u_0$, $\Omega$, $\tilde{\Omega}$, such that
\begin{equation*}\label{eq3.15}
\frac{1}{C_{33}}I_n\leq D^2 u(x,t) \leq C_{33} I_n,\ \ (x,t)\in\bar\Omega_T,
\end{equation*}
where $I_n$ is the $n\times n$ identity matrix.
\end{lemma}

\vspace{3mm}

\section{Longtime existence and convergence}
We will need the following proposition, which essentially asserts the convergence of the flow. Let $\Omega$ be a bounded domain with smooth boundary in $\mathbb{R}^{n}$ and $\mathcal{S}$ be the open connected subset of $\mathbb{S}_{n}$ where
$$  \mathbb{S}_{n}=\{n\times n\,\, \text{real\,\,symmetric\,\,matrix}\}.$$
Given $u_{0}: \Omega\rightarrow\mathbb{R}$,
we consider the parabolic equation with second boundary condition:
\begin{equation}\label{e4.1}
\left\{ \begin{aligned}u_{t}-F(D^{2}u,Du,x)&=0,
&  x\in \Omega,\,\, t>0, \\
h(Du,x)&=0,& x\in \partial\Omega,\,\, t>0,\\
u&=u_{0},& x\in \Omega,\,\, t=0,
\end{aligned} \right.
\end{equation}
where $F: \mathcal{S}\times \mathbb{R}^{n}\times\Omega\rightarrow \mathbb{R}$ is $C^{2+\alpha}$ for some $0<\alpha<1$ and
satisfies
\begin{equation*}\label{e4.2}
A<B\Rightarrow F(A,p,x)<F(B,p,x).
\end{equation*}
\begin{Proposition}\label{p4.1}
(Huang and Ye, see Theorem 1.1 in  \cite{HRY1}.) For any $T>0$, we assume that $u\in C^{4+\alpha,\frac{4+\alpha}{2}}(\bar{\Omega}_{T})$ is a unique solution of the  parabolic equation (\ref{e4.1}), which satisfies
\begin{equation*}\label{e51.4}
\|u_{t}(\cdot,t)\|_{C(\bar{\Omega})}+\|Du(\cdot,t)\|_{C(\bar{\Omega})}+\|D^{2}u(\cdot,t)\|_{C(\bar{\Omega})}\leq \tilde{C}_{1},
\end{equation*}
\begin{equation*}\label{e51.40}
\|D^{2}u(\cdot,t)\|_{C^{\alpha}(\bar{D})}\leq \tilde{C}_{2},\quad \forall D\subset\subset\Omega,
\end{equation*}
and
\begin{equation*}\label{e51.5}
\inf_{x \in \partial \Omega}\big (\sum_{k=1}^{n}h_{p_{k}}(Du(x,t))\nu_{k} \big ) \geq \frac{1}{\tilde{C}_{3}},
\end{equation*}
where the positive constants $\tilde{C}_{1}$, $\tilde{C}_{2}$ and $\tilde{C}_{3}$  are independent of ~ $t\geq 1$. Then the solution $u(\cdot,t)$ converges to a function $u^{\infty}(x,t)=\tilde{u}^\infty(x)+C_{\infty}\cdot t$
in $C^{1+\zeta}(\bar{\Omega})\cap C^{4}(\bar{D})$ as $t\rightarrow\infty$
  for any $D\subset\subset\Omega$, $\zeta<1$, that is
 $$\lim_{t\rightarrow+\infty}\|u(\cdot,t)-u^{\infty}(\cdot,t)\|_{C^{1+\zeta}(\bar{\Omega})}=0,\qquad
  \lim_{t\rightarrow+\infty}\|u(\cdot,t)-u^{\infty}(\cdot,t)\|_{C^{4}(\bar{D})}=0.$$
And $\tilde{u}^{\infty}(x)\in C^{2}(\bar{\Omega})$ is a solution of
\begin{equation}\label{e51.6}
\left\{ \begin{aligned}F(D^{2}u,Du,x)&=C_{\infty},
&  x\in \Omega, \\
h(Du)&=0, &x\in\partial\Omega.
\end{aligned} \right.
\end{equation}
The constant $C_{\infty}$ depends only on $\Omega$ $f$, and $F$. The solution to (\ref{e51.6}) is unique up to additions of constants.
\end{Proposition}

Now, we can give

\noindent{\bf Proof of Theorem \ref{t1.2}.}

This a standard result by our $C^{2}$ estimates and uniformly oblique estimates, but for convenience we include here a proof.

 Part 1: The long time existence.

By Lemma \ref{lem3.6}, we know the global $C^{2,1}$ estimates for the solutions of the flow (\ref{e1.12})-(\ref{e1.14}). Using Theorem 14.22 in Lieberman \cite{GM} and Lemma \ref{llll3.4}, we can show that the solutions of the oblique derivative problem \eqref{e2.16} have global $C^{2+\alpha,1+\frac{\alpha}{2}}$ estimates.

Now let $u_{0}$  be a $C^{2+\alpha_{0}}$ strictly convex function as in the conditions of Theorem \ref{t1.2}. We assume that $T$ is the maximal time such that the solution to the flow \eqref{e2.16} exists. Suppose that $T<+\infty$. Combining Proposition \ref{p4.1} with Lemma \ref{lem3.6} and using Theorem 14.23 in \cite{GM}, there exists $ u\in C^{2+\alpha,1+\frac{\alpha}{2}}(\bar{\Omega}_{T})$ which satisfies \eqref{e2.16} and
$$\|u\|_{C^{2+\alpha,1+\frac{\alpha}{2}}(\bar{\Omega}_{T})}<+\infty.$$
Then we can extend the flow \eqref{e2.16} beyond the maximal time $T$. So that we deduce that $T=+\infty$. Then there exists the solution $u(x,t)$ for all times $t>0$ to (\ref{e1.12})-(\ref{e1.14}).

Part 2: The convergence.

By the boundary condition, we have
\begin{equation*}
\sup_{\bar{\Omega}_{T}}|Du|\leq \tilde{C}_{4},
\end{equation*}
where $\tilde{C}_{4}$ is a constant depending on $\Omega$ and $\tilde{\Omega}$. Using lemma \ref{lem3.6}, it yields
\begin{equation*}\label{e51.4}
\|u_{t}(\cdot,t)\|_{C(\bar{\Omega})}+\|Du(\cdot,t)\|_{C(\bar{\Omega})}+\|D^{2}u(\cdot,t)\|_{C(\bar{\Omega})}\leq \tilde{C}_{5},
\end{equation*}
where the constant $\tilde{C}_{5}$ depending only on $u_0$, $\Omega$, $\tilde{\Omega}$.
By intermediate Schauder estimates for parabolic
equations (cf. Lemma 14.6 and Proposition 4.25 in \cite{GM}), for any $D\subset\subset \Omega$, we have
\begin{equation*}
[D^{2}u]_{\alpha,\frac{\alpha}{2}, D_{T}}\leq C\sup_{\Omega_{T}}|D^{2}u|\leq \tilde{C}_{6},
\end{equation*}
and
\begin{equation*}
\sup_{t\geq 1}\|D^{3}u(\cdot,t)\|_{C(\bar{D})}+\sup_{t\geq 1}\|D^{4}u(\cdot,t)\|_{C(\bar{D})}+\sup_{x_{i}\in D, t_{i}\geq 1}\frac{|D^{4}u(x_{1}, t_{1})-D^{4}u(x_{2}, t_{2})|}{\max\{|x_{1}-x_{2}|^{\alpha},|t_{1}-t_{2}|^{\frac{\alpha}{2}}\}}\leq \tilde{C}_{7},
\end{equation*}
where $ \tilde{C}_{6}$, $\tilde{C}_{7}$ are  constants depending on the known data and dist$(\partial \Omega, \partial D)$.
Using Proposition \ref{p4.1} and combining the bootstrap arguments as in \cite{WHB}, we finish the proof of Theorem \ref{t1.2}.
\qed
\vspace{3mm}

Finally, suppose that each $X(\cdot,t)$ is the graph of a function $u(\cdot,t)$ and $X_{0}=(x,u_{0}(x))$ in $\mathbb{R}^{n+1}$.
We consider the graphic mean curvatture flow in $\mathbb{R}^{n+1}$:
\begin{equation}\label{e4.12}
\frac{\partial u}{\partial t}=\sqrt{1+|Du|^{2}} \mathrm{div}(\frac{Du}{\sqrt{1+|Du|^{2}}}),\quad
 \mathrm{in}\quad \Omega_{T}=\Omega\times(0,T),
\end{equation}
associated with the second boundary value problem
\begin{equation}\label{e4.13}
Du(\Omega)=\tilde{\Omega},  \quad t>0,
\end{equation}
and the initial condition
\begin{equation}\label{e4.14}
u=u_{0},  \quad t=0.
\end{equation}
Based on the same proof of Throrem \ref{t1.2}, an immediate consequence of  the graphic mean curvature flow (\ref{e4.12})-(\ref{e4.14}) for the
second boundary problem  is the following:
\begin{theorem}\label{t4.2}
Assume that $\Omega$, $\tilde{\Omega}$ are bounded, uniformly convex domains with smooth boundary in $\mathbb{R}^{n}$
,$0<\alpha_{0}<1$.
Then for any given initial function $u_{0}\in C^{2+\alpha_{0}}(\bar{\Omega})$
which is   uniformly convex and satisfies $Du_{0}(\Omega)=\tilde{\Omega}$,  the  strictly convex solution of (\ref{e4.12})-(\ref{e4.14}) exists
for all $t\geq 0$ and $u(\cdot,t)$ converges to a function $u^{\infty}(x,t)=u^\infty(x)+C_{\infty}\cdot t$ in $C^{1+\zeta}(\bar{\Omega})\cap C^{4+\alpha}(\bar{D})$ as $t\rightarrow\infty$
for any $D\subset\subset\Omega$, $\zeta<1$,$0<\alpha<\alpha_{0}$, i.e.,

$$\lim_{t\rightarrow+\infty}\|u(\cdot,t)-u^{\infty}(\cdot,t)\|_{C^{1+\zeta}(\bar{\Omega})}=0,\qquad
  \lim_{t\rightarrow+\infty}\|u(\cdot,t)-u^{\infty}(\cdot,t)\|_{C^{4+\alpha}(\bar{D})}=0.$$
And $u^{\infty}(x)\in C^{1+1}(\bar{\Omega})\cap C^{4+\alpha}(\Omega)$ is a solution of
\begin{equation}\label{e1.14a}
\left\{ \begin{aligned}\mathrm{div}(\frac{Du}{\sqrt{1+|Du|^{2}}})&=\frac{C_{\infty}}{\sqrt{1+|Du|^{2}}},
&  x\in \Omega, \\
Du(\Omega)&=\tilde{\Omega}.
\end{aligned} \right.
\end{equation}
The constant $C_{\infty}$ depends only on $\Omega$, $\tilde{\Omega}$. The solution to (\ref{e1.14a}) is unique up to additions of constants.
\end{theorem}

{\bf Acknowledgments:} The authors would like to thank  referees for useful
comments, which improve the paper.


\vspace{5mm}

 \begin{thebibliography}{DU}

\bibitem{HR} R.L. Huang, {\it On the second boundary value problem for Lagrangian mean curvature flow,} J. Funct. Anal. {\bf 269} (2015), 1095-1114.

\bibitem{HRY}R.L. Huang, Y.H. Ye, {\it On the second boundary value problem for a class of fully nonlinear flows I,} Int. Math. Res. Not. {\bf 18}(2019), 5539-5576.


\bibitem{CHY} J.J. Chen, R.L. Huang, Y.H. Ye, {\it On the second boundary value problem for a class of fully nonlinear flows II,}
 Archiv Der Mathematik. {\bf 111}(2018), 1-13.


 \bibitem{SM}S. Brendle, M. Warren,
{\it A boundary value problem for minimal Lagrangian graphs,} J. Differential Geom. {\bf 84} (2010), 267-287.

 \bibitem{KG} K. Ecker and  G. Huisken,  {\it Parabolic methods for the construction of spacelike slices of prescribed mean curvature in cosmological spacetimes,}  Comm. Math. Phys. {\bf135} (1991), No.3, 595-613.

\bibitem{KE}K.Ecker,
{\it Interior estimates and longtime solutions for mean curvature flow of noncompact spacelike hypersurfaces in Minkowski space,} J. Differential Geom. {\bf 46} (1997), no.3, 481-498.

\bibitem{LLJ}
L. Caffarelli, L. Nirenberg, and  J. Spruck, {\it Nonlinear  second  order  equations IV.Starshaped
conipuct  Weingurten hypersurfuces. Current topics in purtiul differential equations,}  ed. by Y.Ohya,K.Kasahara,
N.Shimakura, 1986, pp. 1-26,  Kinokunize Co.,Tokyo.



\bibitem{CT}
Choi. Hyeong,  Treibergs. Andrejs,{\it
Gauss maps of spacelike constant mean curvature hypersurfaces of Minkowski space,}
J. Differential Geom. {\bf 32} (1990), 775-817.


\bibitem{OC} O.C. Schn$\ddot{\text{u}}$rer, {\it Translating solutions to the second boundary value problem
for curvature flows,} Manuscripta Math. {\bf108} (2002), 319-347.

\bibitem{OK} O.C. Schn$\ddot{\text{u}}$rer, K. Smoczyk, {\it Neumann and second boundary value problems for Hessian and  Gauss curvature flows,}
Ann. Inst. H. Poincar$\acute{e}$ Anal. Non Lin$\acute{e}$aire. {\bf 20} (2003), 1043-1073.

\bibitem{JK} J. Kitagawa, {\it A parabolic flow toward solutions of the optimal transportation problem on domains with boundary,}
J. Reine Angew. Math. {\bf672} (2012), 127-160.

\bibitem{MWW} X.N. Ma, P.H. Wang  and W. Wei., {\it Constant mean curvature surfaces and mean curvature flow
with Non-zero Neumann boundary conditions
 on strictly convex domains,} J. Funct. Anal. {\bf 274}(2018), 252-277.

\bibitem{Xin} Y.L. Xin, {\it Mean curvature flow with bounded Gauss image,}  Results. Math. {\bf 59}(2011), no.3-4, 415-436.

\bibitem{MT} M.T. Wang, {\it Gauss maps of the mean curvature flow}, Math. Res. Lett.  {\bf 10}(2003), no.2-3, 287-299.

\bibitem{WHB}C. Wang, R.L. Huang and J.G. Bao, {\it On the second boundary value problem for  a class of fully nonlinear flows III,}
arXiv:2003.04731v1, 2020.



\bibitem{GM} G.M. Lieberman, {\it Second order parabolic differential equations,} World Scientific. 1996.

\bibitem{JS}  J. Spruck , {\it Geometric Aspects of the theory of fully non linear elliptic equations,}  Global Theory of Minimal Surfaces Clay Mathematics Proceedings. 2005:283-309.

\bibitem{J} J. Urbas,  {\it Weingarten hypersurfaces with prescribed gradient image,}  Math. Z. {\bf240} (2002), no. 1, 53-82.

\bibitem{JU} J. Urbas, {\it On the second boundary value problems for equations of Monge-Amp\`{e}re type,} J. Reine Angew. Math. {\bf 487} (1997), 115-124.

\bibitem{LL1} Y.Y. Li, L. Nirenberg, {\it A geometric problem and the Hopf lemma.I,} J. Eur. Math. Soc. (JEMS) {\bf8}(2006), no.2, 317-339.
\bibitem{LL2} Y.Y. Li, L. Nirenberg, {\it A geometric problem and the Hopf lemma.II,} Chinese Ann. Math. Ser. B {\bf27}(2006), no.2, 193-218.

\bibitem{RS} R.L. Huang and S.T. Li, {\it On the second boundary value problem for special Lagrangian curvature potential equation,}
Math. Z. {\bf302} (2022), no.1, 391-417.

\bibitem{WHB1}C. Wang, R.L. Huang and J.G. Bao, {\it On the second boundary value problem for Lagrangian mean curvature equation,}
 Calc. Var. Partial Differential Equations {\bf62}(2023), no.3, Paper No.74, 30 pp.

 \bibitem{GH} G. Huisken,  {\it Flow by mean curvature of convex surfaces into spheres,} J. Differential Geom. {\bf20} (1984), No.1, 237-266.

\bibitem{RSH} R.S. Hamilton,  {\it Three-manifolds with positive Ricci curvature,} J. Differential Geom. {\bf17} (1982), No.2, 255-306.

\bibitem{HRY1} R.L. Huang, Y.H. Ye, {\it A convergence result on the second boundary value problem for parabolic equations,}
 Pacific J. Math. {\bf310} (2021), No.1, 159-179.














































\end{thebibliography}
\end{document}